\documentclass[11pt,oneside,reqno]{amsart}
\usepackage{mathpazo} 
\normalfont
\usepackage[T1]{fontenc}

\usepackage[utf8]{inputenc}
\usepackage{amsmath}
\usepackage{amsfonts}
\usepackage{amssymb}
\usepackage{amsthm}
\usepackage{mathtools}
\usepackage{stmaryrd}
\usepackage{tikz}
\usepackage{tikz-cd}
\usepackage[all]{xy}
\usepackage{enumitem}
\usepackage[margin=1.2in]{geometry} 
\usepackage{multicol}

\makeatletter
\let\origsection\section
\renewcommand\section{\@ifstar{\starsection}{\nostarsection}}

\newcommand\nostarsection[1]
{\sectionprelude\origsection{#1}\sectionpostlude}

\newcommand\starsection[1]
{\sectionprelude\origsection*{#1}\sectionpostlude}

\newcommand\sectionprelude{%
  \vspace{1em}
}

\newcommand\sectionpostlude{%
  \vspace{1em}
}
\makeatother

\usepackage{skak}
\usepackage{accents}

\newtheorem{thm}{Theorem}[section]
\newtheorem{theorem}[thm]{Theorem}
\newtheorem{thma}{Theorem}

\newcommand{\repthm}[2]{
  \newtheorem*{irepthm}{Theorem #1}
  \begin{irepthm}
  #2
  \end{irepthm}
}
\newtheorem{lemma}[thm]{Lemma}

\newtheorem{coro}[thm]{Corollary}

\newtheorem{prop}[thm]{Proposition}
\newtheorem{proposition}[thm]{Proposition}

\theoremstyle{definition}
\newtheorem{defi}[thm]{Definition}
\newtheorem{definition}[thm]{Definition}
\newtheorem{remark}[thm]{Remark}
\newtheorem{example}[thm]{Example}
\newtheorem{notation}[thm]{Notation}

\theoremstyle{plain}

\newcommand\xqed[1]{%
  \leavevmode\unskip\penalty9999 \hbox{}\nobreak\hfill
  \quad\hbox{#1}}
\newcommand\exEnd{\xqed{$\diamondsuit$}}

\newcommand{\B}{\mathbb{B}}

\newcommand{\R}{\mathbb{R}}
\newcommand{\T}{\mathbb{T}}
\newcommand{\TT}{\T}
\newcommand{\Z}{\mathbb{Z}}
\newcommand{\ZZ}{\Z}

\newcommand{\Q}{\mathbb{Q}}
\newcommand{\Qmax}{\mathbb{Q}_{\mathrm{max}}}

\newcommand{\calG}{\mathcal{G}}
\newcommand{\calg}{\calG}
\newcommand{\calC}{\mathcal{C}}
\newcommand{\calc}{\calC}
\newcommand{\calP}{\mathcal{P}}
\newcommand{\calp}{\calP}
\newcommand{\calF}{\mathcal{F}}
\newcommand{\calf}{\calF}
\newcommand{\calD}{\mathcal{D}}
\newcommand{\cald}{\calD}
\newcommand{\calQ}{\mathcal{Q}}
\newcommand{\calq}{\calQ}
\newcommand{\calH}{\mathcal{H}}
\newcommand{\calh}{\calH}
\newcommand{\calB}{\mathcal{B}}
\newcommand{\calb}{\calB}
\newcommand{\fraku}{\mathfrak{u}}
\newcommand{\frakS}{\mathfrak{S}}

\newcommand{\inj}[0]{\ar@{^{(}->}}
\newcommand{\surj}[0]{\ar@{->>}}
\newcommand{\bij}[0]{\ar@{^{(}->>}}
\newcommand{\lbij}[0]{\ar@{_{(}->>}}
\newcommand{\linj}[0]{\ar@{_{(}->}}
\newcommand{\parr}[0]{\ar@{.>}}
\newcommand{\pinj}[0]{\ar@{^{(}.>}}
\newcommand{\psurj}[0]{\ar@{.>>}}
\newcommand{\pbij}[0]{\ar@{^{(}.>>}}
\newcommand{\lin}[0]{\ar@{-}}
\newcommand{\llin}[0]{\ar@{=}}
\newcommand{\usu}[0]{\ar@{->}}
\newcommand{\cmapsto}{\ar@{|->}}

\newcommand{\dsum}{\displaystyle\sum}
\newcommand{\dprod}{\displaystyle\prod}
\newcommand{\dcup}{\displaystyle\bigcup}

\newcommand{\sdrop}{\backslash}
\newcommand{\into}{\hookrightarrow}
\newcommand{\eps}{\varepsilon}
\newcommand{\ph}{\varphi}
\newcommand{\angbra}[1]{\left\langle #1\right\rangle}

\newcommand{\Terms}[1]{\operatorname{Terms}(#1)}

\newcommand{\nocontentsline}[3]{}
\newcommand{\tocless}[2]{\bgroup\let\addcontentsline=\nocontentsline#1{#2}\egroup}

\newcommand{\Cont}{\operatorname{Cont}}
\newcommand{\ContInt}{\Cont^{\circ}}
\newcommand{\ContBase}[1]{\Cont_{#1}}
\newcommand{\ContBaseInt}[1]{\ContInt_{#1}}
\newcommand{\ContIntBase}[1]{\ContBaseInt{#1}}

\def\Mon{M}
\def\base{S}

\newcommand{\GeorgeStory}[1]{}
\newcommand{\Hom}{\operatorname{Hom}}
\newcommand{\RnLexSemifield}[1]{\T^{(#1)}}
\newcommand{\rk}{\operatorname{rk}}

\newcommand{\Paul}{\mathfrak{P}}
\newcommand{\BigLarry}{lattice of $\Gamma$-rational polyhedral sets}
\newcommand{\LittleLarry}{lattice of $\Gamma$-rational polytopal sets}
\newcommand{\ConoidLarry}{lattice of $\Gamma$-admissible \ConoidSets}

\newcommand{\F}{\mathcal{F}}
\newcommand{\trdeg}{\operatorname{tr.deg}}

\newcommand{\HT}{\operatorname{ht}}
\newcommand{\mbo}[1]{\pmb{#1}}

\newcommand{\relint}{\operatorname{rel.int.}}
\newcommand{\dcap}{\displaystyle\bigcap}
\newcommand{\conoidSet}{fan support set}

\newcommand{\conoidSets}{fan support sets}
\newcommand{\ConoidSets}{\conoidSets}

\newcommand{\wt}[1]{\widetilde{#1}}
\newcommand{\wtcalf}{\wt{\calf}}
\newcommand{\lex}{\mathrm{lex}}

\newcommand{\Grobner}{Gr\"{o}bner }
\newcommand{\grobner}{\Grobner}

\newcommand{\ubar}[1]{\underline{#1}}

\usepackage{lipsum}

\makeatletter
\renewcommand\subsection{\@startsection{subsection}{2}%
  \z@{-.5\linespacing\@plus-.7\linespacing}{.5\linespacing}%
  {\centering\normalfont\scshape}}
\renewcommand\subsubsection{\@startsection{subsubsection}{3}%
  \z@{.5\linespacing\@plus.7\linespacing}{-.5em}%
  {\normalfont\scshape}}
\makeatother

\makeatletter
\let\origsubsection\subsection
\renewcommand\subsection{\@ifstar{\starsubsection}{\nostarsubsection}}

\newcommand\nostarsubsection[1]
{\subsectionprelude\origsubsection{#1}\subsectionpostlude}

\newcommand\starsubsection[1]
{\subsectionprelude\origsubsection*{#1}\subsectionpostlude}

\newcommand\subsectionprelude{%
  \vspace{0.5em}
}

\newcommand\subsectionpostlude{%
  \vspace{0.5em}
}
\makeatother

\title{Geometric interpretation of valuated term (pre)orders}

\author{Netanel Friedenberg }
\address{Department of Mathematics, Tulane University, New Orleans, LA 70118, USA}
\email{nfriedenberg@tulane.edu}

\author[K.~Mincheva]{Kalina~Mincheva}
\address{Department of Mathematics, Tulane University, New Orleans, LA 70118, USA}
\email{kmincheva@tulane.edu}

\begin{document}

\begin{abstract}
    Valuated term orders are studied for the purposes of Gr\"{o}bner theory over fields with valuation.
    The points of a usual tropical variety correspond to certain valuated terms preorders. Generalizing both of these, the set of all ``well-behaved'' valuated term preorders is canonically in bijection with the points of a space introduced in our previous work on tropical adic geometry. In this paper we interpret these points geometrically by explicitly characterizing them in terms of classical polyhedral geometry. This characterization gives a bijection with equivalence classes of flags of polyhedra as well as a bijection with a class of prime filters on a lattice of polyhedral sets. The first of these also classifies valuated term orders. The second bijection is of the same flavor as the bijections from \cite{PS95}
    in non-archimedean analytic geometry and indicates that the results of that paper may have analogues in tropical adic geometry. 
\end{abstract}
\maketitle

\section{Introduction}


Classical \grobner theory uses a term order which compares two monomials $m = a\ubar{x}^{\ubar{u}}$ and $m' = a'\ubar{x}^{\ubar{u'}}$ by only considering $\ubar{u}$ and $\ubar{u'}$. In \cite{CM19} Chan and Maclagan build Gr\"{o}bner theory in $k[x_1,\ldots,x_n]$, when $k$ is a field with nontrivial valuation $v:k\to\R\cup\{-\infty\}=:\T$. Their valuated \Grobner theory takes into account the coefficients by picking a weight vector $\ubar{w}\in\R^n$ and comparing $v(a)+\ubar{w}\cdot \ubar{u}$ with $v(a')+\ubar{w}\cdot\ubar{u}'$, and then using a term order to break any ties.

Valuated \grobner theory has a natural connection to tropical geometry. More specifically, the tropicalization of a subvariety of the torus is the set of $\ubar{w} \in \R^n$ for which each polynomial in its defining ideal has at least two monomials of maximal weight. In particular, tropical geometry studies valuated monomial preorders, rather than orders. A \emph{valuated monomial preorder} is a multiplicative total preorder $\leq$ on the set of monomials $a\ubar{x}^{\ubar{u}}$ such that if $v(a) \leq v(b)$ then $a \leq b$. Valuated monomial \emph{orders} are those for which $a\ubar{x}^{\ubar{u}}\equiv b\ubar{x}^{\ubar{\mu}}$ if and only if $v(a)=v(b)$ and $u=\mu$. In light of this, we treat $a\ubar{x}^{\ubar{u}}$ and $b\ubar{x}^{\ubar{u}}$ as being the same if $v(a)=v(b)$.

When working with valuated monomial orders, one is usually only interested in those in which the valuation of the coefficients play a primary role. These are exactly the orders that arise by using a weight vector and term order as above. We call such orders Chan-Maclagan orders; they also occur in ongoing work of Vaccon and Verron \cite{VV} on universal \grobner bases in Tate algebras.

Recent work of Amini and Iriarte \cite{AI22} aimed at the study of Newton-Okounkov bodies investigates spaces of (quasi)monomial preorders. An understanding of valuated monomial preorders is needed to extend their theory from varieties to semistable models over a valuation ring.

We see from the above that a weight vector $\ubar{w}\in\R^n$ provides a geometric interpretation for certain valuated monomial preorders.
In fact, if $v(k^\times)\subseteq\Q$, then most points in $\R^n$ define a valuated monomial \emph{order}. For if $\ubar{w}\in\R^n$ has irrational entries which are linearly independent over $\Q$, then distinct monomials will always have different weights.
This shows that certain valuated monomial orders have a geometric interpretation; however, not all of them arise this way.

The goal of the paper is to derive a geometric meaning for each of the following: (1) all valuated monomial preorders, (2) Chan-Maclagan orders, and (3) those valuated monomial preorders determined by a weight vector and a term \emph{preorder} as in the construction of the Chan-Maclagan orders. We do this as an application of the theory proposed in \cite{FM22}, thus verifying that 
it
carries geometric information.

The theory from \cite{FM22} applies more generally, and so we work with $k[\Mon]$ where $\Mon$ is a toric monoid. We are then able to quickly reduce to the torus case, where $k[\Mon]=k[x_1^{\pm1},\ldots,x_n^{\pm1}]$.

Let $\Gamma=v(k^\times)$ be the value group of $v$. We can now state 
our main results.

\begin{thma}
[Corollary~\ref{coro:explicit BijectionsAllPrimeCongsAndThreeOthers}]\label{thm:ClassifyAllPreorders}

There is an explicit bijection from the set of valuated monomial preorders on $k[x_1^{\pm1},\ldots,x_n^{\pm1}]$ to the set of $\Gamma$-local equivalence classes of flags of polyhedral cones in $\R_{\geq0}\times\R^n$.
\end{thma}

See Definition~\ref{def: local-equiv} for the definition of $\Gamma$-local equivalence.

\begin{thma}
[Corollary~\ref{coro:explicitBijectionsContAndThreeOthers}]\label{thm:ClassifyCMPreorders}

There is an explicit bijection from the set of Chan-Maclagan preorders on $k[x_1^{\pm1},\ldots,x_n^{\pm1}]$ to the set of $\Gamma$-local equivalence classes of flags of polyhedra in $\R^n$.
\end{thma}

\begin{thma}[Proposition~\ref{prop:whenDoesAPrimeTotallyOrder}]\label{thm:IntroClassifyCMOrders}
The bijection of Theorem \ref{thm:ClassifyCMPreorders} gives a bijection from the set of Chan-Maclagan orders on $k[x_1^{\pm1},\ldots,x_n^{\pm1}]$ to the set of $\Gamma$-local equivalence classes of complete flags of polyhedra in $\R^n$.
\end{thma}


The classification provided by Theorem~\ref{thm:IntroClassifyCMOrders} can be seen as complementary to an earlier result in the non-valuated case.
Usual 
monomial orders have been classified in \cite{Rob85} in terms of their defining matrices. Since two matrices can give the same monomial order, the classification proceeds by providing a canonical representative for the equivalence class. In contrast, Corollary~\ref{coro:DoTheseMatricesGiveTheSamePrime} gives an alternative, geometric criterion for when two matrices give the same valuated or non-valuated monomial order.

We now interpret the Chan-Maclagan preorders in the context of (additively idempotent) semirings, motivated by adic geometry.
We first recall that for a topological (usually f-adic) ring $R$  the set of equivalence classes of continuous valuations $v:R\to G\cup\{0\}$, where $G$ can be any multiplicatively-written totally ordered abelian group, is called the continuous spectrum of $R$, denoted $\operatorname{Cont}R$. To consider continuity of valuations, a certain canonical topology is placed on $G\cup\{0\}$.

The semiring analogue of a valuation is a homomorphism to a totally ordered semifield.
Just as equivalence classes of homomorphisms from a ring $R$ to a field are given by prime ideals, equivalence classes of valuations on a semiring $A$ are given by certain equivalence relations on $A$ called \emph{prime congruences}\footnote{first introduced in \cite{JM17}}.

For any topological semiring $A$ we consider the set of prime congruences $P$ on $A$ for which $A/P$ is not finite (a technical hypothesis that is redundant in the cases we consider) and the map from $A$ to the residue semifield $\kappa(P)$ is continuous. Here $\kappa(P)$, the semifield generated by $A/P$, is given a canonical topology. We call the set of such $P$ the \emph{continuous spectrum of $A$}, and denote it by $\Cont A$. These spaces are directly linked to the spaces that occur in adic geometry: a continuous generalized valuation in the sense of \cite[Definition 2.5.1]{GG16} from a ring $R$ to a semiring $A$ induces a map from $\Cont{A}$ to $\Cont{R}$.

In this language, the set of Chan-Maclagan preorders on $k[x_1^{\pm1},\ldots,x_n^{\pm1}]$ is canonically identified with the continuous spectrum of $\base[x_1^{\pm1},\ldots,x_n^{\pm1}]$. Here $\base$ is the set $\Gamma\cup\{-\infty\}$ endowed with operations that make it a totally ordered semifield, and $\base[x_1^{\pm1},\ldots,x_n^{\pm1}]$ is given a topology coming from the inclusion map $\base\into\base[x_1^{\pm1},\ldots,x_n^{\pm1}]$. We denote this continuous spectrum by $\ContBase{\base}\base[x_1^{\pm1},\ldots,x_n^{\pm1}]$. Thus Theorem~\ref{thm:ClassifyCMPreorders} can be restated as follows.
\repthm{\ref{thm:ClassifyCMPreorders}'}{
There is an explicit bijection from $\ContBase{\base}\base[x_1^{\pm1},\ldots,x_n^{\pm1}]$ to the set of $\Gamma$-local equivalence classes of flags of polyhedra in $\R^n$.}

The motivation to interpret certain valuated monomial preorders as points on a tropical adic space, and thus prime congruences, comes from the connection between tropical geometry and Berkovich spaces. It is also inspired by the broader program to endow tropical varieties with more structure; in \cite{FM22} we take an analytic approach to doing so, which led us to the results in this work. There are alternative algebraic approaches using many different frameworks. Among these are blueprints \cite{Lor15},  tropical ideals \cite{MR18}, tropical schemes \cite{GG16}, super-tropical algebra \cite{IR10}, and systems \cite{R18}. 

Our last major result gives an alternative classification of the points of $\ContBase{\base}\base[x_1^{\pm1},\ldots,x_n^{\pm1}]$. This classification is reminiscent of a theorem for adic spaces.

\begin{thma}[Theorem~\ref{thm:BijectionFiltersAndCont}]
There is an explicit bijection between $\ContBase{\base}\base[x_1^{\pm1},\ldots,x_n^{\pm1}]$ and the set of prime filters $\calf$ on the \BigLarry\ in $\R^n$ such that $\calf$ contains a polytope.
\end{thma}

This bijection between points of $\Cont_\base\base[x_1^{\pm1},\ldots,x_n^{\pm 1}]$ and prime filters on the \BigLarry\ is of the same flavor as certain results in \cite{PS95}. If $X$ is a rigid affinoid space, then it is shown in \cite{PS95} that the points of the Huber adic space corresponding to $X$ are in bijection with the prime filters on the lattice of special subsets of $X$.

We also show that, for any $P\in\Cont_\base\base[x_1^{\pm1},\ldots,x_n^{\pm1}]$, the minimum dimension of any element of the corresponding filter is equal to the transcendence degree of the semifield extension $\kappa(P)/\base$; see Proposition~\ref{prop: min-dim}.

\addtocontents{toc}{\protect\setcounter{tocdepth}{-1}}
\section*{Acknowledgements}

The authors thank Hernán Iriarte for exciting discussions and Bernd Sturmfels for pointing us to the work of Vaccon and Verron in valuated Gr\"{o}bner theory. 
The second author is partially supported by Louisiana Board of Regents Targeted Enhancement Grant number 090ENH-21.

\addtocontents{toc}{\protect\setcounter{tocdepth}{2}}

\section{Preliminaries}\label{sect:prelim}

\begin{definition}

By a {\it semiring} $R$ we mean a commutative semiring with multiplicative unit. That is, $R$ is a set which is a commutative monoid with respect to each of two binary operations: an addition operation $+_R$, whose identity is denoted $0_R$, and a multiplication operation $\cdot_R$, whose identity is denoted $1_R$. Furthermore, we require that $0_R\neq1_R$, $0_R$ is multiplicatively absorbing, and multiplication distributes over addition. We omit the subscripts from the operations whenever this will not cause ambiguity. A {\it semifield} is a semiring in which all nonzero elements have a multiplicative inverse.
\end{definition}

\begin{defi}
We call a semiring $R$ {\it additively idempotent} if $a+a = a$ for all $a \in R$. We refer to additively idempotent semirings as just {\it idempotent}.
\end{defi}

\noindent If $R$ is an idempotent semiring, then the addition defines a canonical partial order on $R$ in the following way: $a \geq b \Leftrightarrow a + b = a$.
With respect to this order, $a+b$ is the least upper bound of $a$ and $b$. When we consider a totally ordered idempotent semiring, we mean an idempotent semiring for which this canonical order is a total order. 
\par\medskip
\noindent Some of the idempotent semifields that we use throughout the paper are:
\begin{itemize}
    \item The {\it Boolean semifield} $\mathbb{B}$ is the semifield with two elements $\{1,0\}$, where $1$ is the multiplicative identity, $0$ is the additive identity and $1+1 = 1$. $\B$ is the only finite additively idempotent semifield. 
    \item The {\it tropical semifield} $\T$ is defined on the set $\R  \cup \{-\infty\} $, by setting the $+$ operation to be the usual maximum and the multiplication operation to be the usual addition, with $-\infty = 0_\T$. 
    \item The semifield $\Q_{\max}$ is a sub-semifield of $\T$. As a set it is $\Q \cup \{-\infty\} $ and the operations are the restrictions of the operations on $\T$. We also use the same notation when $\Q$ is replaced by any other additive subgroup of $\R$. In fact, every sub-semifield of $\T$ arises this way.  
    \item The semifield $\RnLexSemifield{n}$ is defined on the set $\R^n  \cup \{-\infty\} $, by setting the addition operation to be lexicographical maximum and the multiplication operation to be the usual pointwise addition.
\end{itemize}

Throughout this paper all monoids, rings and semirings are assumed to be commutative. All rings and semirings have a multiplicative unit. Whenever we use the word semiring without further qualification, we refer to an additively idempotent semiring.
Our focus will be on totally ordered semifields which are different from $\B$. Henceforth, when we refer to totally ordered semifields, we implicitly assume that they are not isomorphic to $\B$. In particular, whenever we refer to a sub-semifield of $\T$, we implicitly assume that it is not $\B$. Totally ordered semifields can be seen as the image of a non-archimedean valuation. The semifield $\B$ then is the image of the trivial valuation.

In order to distinguish when we are considering a real number as being in $\R$ or being in $\T$ we introduce some notation. For a real number $a$, we let $t^a$ denote the corresponding element of $\T$. In the same vein, given $\mathfrak{a}\in\T$, we write $\log(\mathfrak{a})$ for the corresponding element of $\R\cup\{-\infty\}$. This notation is motivated as follows.
Given a non-archimedean valuation $\nu: K \rightarrow \R\cup\{-\infty\}$ on a field and $\lambda \in \R$ with $\lambda>1$, we get a non-archimedean absolute value $|\cdot|_\nu:K\rightarrow[0, \infty)$ by setting $|x|_{\nu}=\lambda^{\nu(x)}$. Since $\T$ is isomorphic to the semifield $\big([0,\infty),\max,\cdot_{\R}\big)$, we use a notation for the correspondence between elements of $\R\cup\{-\infty\}$ and elements of $\T$ that is analogous to the notation for the correspondence between $\nu(x)$ and $|x|_{\nu}$. This notation is also convenient as we get many familiar identities such as $\log(1_{\T})=0_{\R}$ and $t^a t^b=t^{a+_{\R}b}$.

\begin{defi}
Let $R$ be a semiring and let $a \in R\sdrop\{0\}$. We say that $a$ is a \emph{cancellative element} if for all $b, c \in R$, whenever $ab=ac$ then $b=c$. If all elements of $R\sdrop\{0\}$ are cancellative, then we say that $R$ is a \textit{cancellative semiring}.
\end{defi}

\begin{definition}[Congruences]\label{def: prime_domain}\ 

\begin{itemize}

\item A \textit{congruence} $E$ on a semiring $R$ is an equivalence relation on $R$ that respects the operations of $R$.

\item The \emph{trivial congruence} on $R$ is the diagonal $\Delta\subseteq R\times R$, for which $R/\Delta\cong R$.

\item We call a proper congruence $P$ of a semiring $R$  {\it prime} if $R/P$ is totally ordered and cancellative (cf.\ Definition 2.3 and Proposition 2.10 in \cite{JM17}).


\end{itemize}

\end{definition}

We will need some more notation around prime congruences.

\begin{defi}\label{def:ResidueSemifield}
Let $P$ be a prime congruence on a semiring $A$. The {\it residue semifield of $A$ at $P$}, denoted $\kappa(P)$, is the total semiring of fractions\footnote{The total semiring of fractions is a classical concept and is defined in \cite{Gol92}. An alternative reference following the notation of this paper is in our previous work \cite[Definition 2.13]{FM22}.} of $A/P$. We denote the canonical homomorphism $A\to A/P\to\kappa(P)$ by $a\mapsto|a|_P$.
\end{defi}

\begin{notation}
Let $R$ be semiring and let $P$ be a prime on $R$. For two elements $r_1, r_2 \in R$ we say that $r_1 \leq_P r_2$ (resp. $r_1 <_P r_2$, resp. $r_1 \equiv_P r_2$) whenever $|r_1|_P \leq |r_2|_P$ (resp. $|r_1|_P <  |r_2|_P$, resp. $|r_1|_P = |r_2|_P$) in $R/P$.
\end{notation}


We now focus our attention on the case when $R$ is a monoid algebra. Given a totally ordered semifield $\base$ and a monoid $\Mon$, the corresponding monoid $\base$-algebra is denoted $\base[\Mon]$. Elements of $\base[\Mon]$ are finite sums of expressions of the form $a\chi^u$ with $a\in\base$ and $u\in\Mon$, which are called \emph{monomials} or \emph{terms}. We call elements of $\base[\Mon]$ \emph{generalized semiring polynomials} or \emph{polynomials} when no confusion will arise. We will be most interested in $\base[\Z^n]$, the \emph{semiring of tropical Laurent polynomials in $n$ variables}.

\medskip
When $\base\subseteq \TT$ we describe the prime congruences on $\base[\ZZ^n]$ in terms of their \emph{defining matrices} and the monomial preorder the give rise to. Every prime congruence on $\base[\Z^n]$ has a defining matrix; see \cite{JM17}, where it is shown that the matrix can be taken to be particularly nice. 

Let $R = \base[\ZZ^n]$ where $\base$ is a sub-semifield of $\TT$. A $k\times (n+1)$ real valued matrix $C$ such that the first column of $C$ is lexicographically greater than or equal to the zero vector gives a prime congruence on $R$ as follows:
For any monomial $m=t^a\chi^u\in\base[\Mon]$ we let $\Phi(m)=C\begin{pmatrix}a \\ u\end{pmatrix}\in\R^k$, where we view $u\in\Z^n$ as a column vector. We call $\begin{pmatrix}a \\ u\end{pmatrix}$ the \emph{exponent vector} of the monomial $m$. For any nonzero $f\in\base[\Mon]$, write $f$ as a sum of monomials $m_1,\ldots,m_r$ and set $\Phi(f)=\max_{1\leq i\leq r}\Phi(m_i)$, where the maximum is taken with respect to the lexicographic order. Finally, set $\Phi(0)=-\infty$. We specify the prime congruence $P$ by saying that $f$ and $g$ are equal modulo $P$ if $\Phi(f)=\Phi(g)$. In this case we say that $C$ is a defining matrix for $P$. The first column of $C$ is called the column corresponding to the coefficient or the column corresponding to $\base$. For $1\leq i\leq n$, we say that the $(i+1)^{\text{st}}$ column of $C$ is the column corresponding to $x_i$.


We can also explicitly describe the order that $P$ gives rise to. Given any defining matrix for $P$ and any $f,g\in R$, $f\leq_{P}g$ exactly if $\Phi(f)\leq_{lex}\Phi(g)$.


It will be convenient to know that we can choose a defining matrix with a particular form. To accomplish this, we will use the following lemma.

\begin{lemma}[\cite{FM22} Lemma 2.12, Downward gaussian elimination]\label{lemma:rref}
Let $\base$ be a sub-semifield of $\T$ and $\Mon=\ZZ^n$ and let $P$ be a prime congruence of $\base[\Mon]$ with defining matrix $C$. Then the following elementary row operations do not change the prime.
\begin{itemize}
    \item multiplying a row by a positive constant.
    \item adding a multiple of a row to any row below it. 
\end{itemize}
\end{lemma}

Note that, if a row of the matrix is all zero, then removing the row does not change the prime congruence that the matrix defines. Thus, using downward gaussian elimination, we can always choose a defining matrix in which the rows are linearly independent (over $\R$).


In \cite{FM22} we define a set $\ContBase{\base} A$ of certain prime congruences on $A$. When $A=\base[\Mon]$ with $M=\Z^n$, the following proposition completely characterizes $\ContBase{\base} A$.

\begin{proposition}
[\cite{FM22} Corollary 4.16]\label{coro: matrix_form}
Let $\base$ be a sub-semifield of $\T$ and let $\Mon\cong\Z^n$. Let $P$ be a prime on $\base[\Mon]$ and let $C$ be any defining matrix for $P$. Then $P\in\ContBase{\base}\base[\Mon]$ if and only if the $(1,1)$ entry of $C$ is positive. In particular, if $P\in\ContBase{\base}\base[\Mon]$, then we can choose a defining matrix for $P$ whose first column is the first standard basis vector $e_1$.
\end{proposition}


By $\Terms{\base[\Mon]}$ we denote the multiplicative monoid of (nonzero) terms $a\chi^u$ of $\base[\Mon]$; if $\Mon$ is a group, then $\Terms{\base[\Mon]}$ is a group.

\begin{proposition}[\cite{FM22} Proposition 9.10]\label{prop: side_claim}
Let $\Mon\cong\Z^n$. A relation $\preceq$ on $\Terms{\base[\Mon]}$ is the same as $\leq_P$ for a prime $P$ on $\base[\Mon]$ if and only if the following conditions hold:
\begin{enumerate}
    \item[(1)] $\preceq$ is a total preorder, 
    \item[(2)] $\preceq$ is multiplicative, and
    \item[(3)] $\preceq$ respects the order on $\base$.
\end{enumerate}
Moreover, $P\in\ContBase{\base}{\base[\Mon]}$ if and only if
\begin{enumerate}[resume]
    \item[(4)] for any term $a\chi^u$ of $\base[\Mon]$, $\exists b \in \base^\times$, such that $b \prec a\chi^u$.
\end{enumerate}
\end{proposition}

\begin{defi}
The \emph{height} of a prime congruence $P$ on a semiring $A$ is the maximum length $\HT(P)$ of a chain $P_0\subsetneq P_1\subsetneq\cdots\subsetneq P_{k}=P$ of primes under $P$.
\end{defi}

In \cite{JM17} the \emph{dimension} of a semiring $A$, denoted $\dim A$, is defined to be the number of strict inclusions in a chain of prime congruences on $A$ of maximal length. We refine this notion in the next definition. 

\begin{defi}
When $\base \subseteq \T$ is a sub-semifield and $A$ is a $\base$-algebra, the \emph{relative dimension of $A$ over} $\base$ is the number $\dim_\base A$ of strict inclusions in a longest chain of prime congruences in $\ContBase{\base}{A}$.

\end{defi}

We illustrate the distinction between dimension and relative dimension in the following example.

\begin{example}
Let $\base$ be a sub-semifield of $\T$. Then $\dim \base$ is equal to the number of proper convex subgroups of $\base^\times $ which is 1. To see this, note that $\dim \base$ is the largest possible length of a chain of semiring quotients of $\base$. Since every quotient of a totally ordered semifield is a totally ordered semifield, and since $\base\mapsto\base^\times$ gives an equivalence between the category of totally ordered semifields and the category of totally ordered abelian groups \cite[Theorem A.4]{FM22} this is the same as the largest possible length of a chain of totally ordered abelian group quotients of $\base^\times$. But such quotients are given by convex subgroups of $\base^\times$ and these are totally ordered by containment, so $\dim\base$ is the number of such subgroups. On the other hand, $\dim_\base \base=0$.\exEnd\end{example}


We now briefly introduce some notation from toric varieties. In this paper we follow conventions from \cite{Rab10}. We will use the notation to make a statement about the dimension of the semiring $\base[\Mon]$ and its geometric interpretation. 

Let $\Lambda$ be a finitely generated free abelian group and let $\Lambda_\R $ be the vector space $ \Lambda\otimes_{\Z}\R$. We write $N:=\Lambda^*=\Hom(\Lambda,\Z)$ and $N_\R=\Hom(\Lambda,\R)\cong N\otimes_{\Z}\R$. We will use the pairing $N_\R\times\Lambda_{\R}\to\R$ given by $(v,u)\mapsto\angbra{v,u}:=v(u)$. 
A cone $\sigma$ is a \textit{strongly convex rational polyhedral cone} in $N_\R$ if $\sigma = \sum\limits_{i=1}^{r}\R_{\geq 0}v_i$ for $v_i \in N$ and $\sigma$ contains no line. 
We denote by $\sigma^\vee$ the dual cone of $\sigma$,
$$\sigma^\vee = \{u\in \Lambda_\R : \left< v,u \right> \leq 0, \forall v\in \sigma\} = \bigcap\limits_{i = 1}^r \{u\in \Lambda_\R : \left< v_i,u \right> \leq 0\}.$$

The monoids that we consider in this paper are of the form $ M = \sigma^\vee \cap \Lambda$. Because $\sigma$ is strongly convex, we can identify $\Lambda$ with the groupification of $M$.

\begin{lemma}[\cite{FM22} Lemma 9.13]\label{lemma:HeightInToricCase}
Let $\base$ be  sub-semifield of $\T$ and let $\Mon\subseteq\Lambda$ be a toric monoid corresponding to a cone $\sigma$ in $N_{\R}$. For any $P\in\ContIntBase{\base}\base[\Mon]$, $\HT(P)=\rk(\Lambda)-\rk(\kappa(P)^\times/\base^\times)$.
\end{lemma}

\begin{proposition}[\cite{FM22} Proposition 9.18]\label{prop:rk-of-lattice-dim-of-Ssigma}
Let $\base$ be a sub-semifield of $\T$, let $\Lambda$ be a finitely generated free abelian group, let $\sigma$ be a cone in $N_{\R}$, and let $\Mon = \sigma^\vee\cap\Lambda$. Then $\dim_\base\base[\Mon]=\rk \Lambda.$
\end{proposition}


\section{Results}\label{sect: geom interpret}

We provide a geometric interpretation of the points of $\ContBase{\base}{\base[\Mon]}$ with $\base$ a sub-semifield of $\T$ and $\Mon$ a toric monoid. We will denote by $\Gamma:=\log(\base^\times)$ the subgroup of $\R$ corresponding to $\base$.

By \cite[Proposition 9.14]{FM22}, fixing the ideal-kernel of $P\in\ContBase{\base}{\base[\Mon]}$ decomposes $\ContBase{\base}{\base[\Mon]}$ into toric strata, which are tori. Thus, we only need to deal with the torus case, i.e., the case where $\Mon$ is a finitely generated free abelian group. 

The main results of this paper show that, in this case, there are explicit bijections between the points of $\ContBase{\base}{\base[\Mon]}$, the set of prime filters on the \BigLarry\ in $N_{\R}$ where the filter contains a polytope, and certain equivalence classes of flags of polyhedra.

\subsection{Flags and prime congruences.}
By a \emph{cone} we will mean a strongly convex polyhedral cone which is not necessarily rational.
We note that there is a bijective correspondence between the set of non-empty strongly convex polyhedra  $\mathcal{P} \subseteq N_\R$ and the set of strongly convex cones $\sigma \subseteq \R_{\geq 0} \times N_\R$, not contained in $\{0\} \times N_\R$. Under this correspondence we map a polyhedron $\mathcal{P}$ to the closed cone over it, denoted $c(\mathcal{P})$, and a cone $\calc$ to the restriction of $\calc$ to $\{1\} \times N_\R$, by which we mean the inverse image of $\calc$ under the map $N_{\R}\to\R_{\geq0}\times N_{\R}$ given by $x\mapsto(1,x)$. 

Additionally, we can think of a point $ w=(r,\xi)\in\R_{\geq0}\times N_{\R}$ as a homomorphism $ w:\Gamma\times \Mon\to\R$, mapping $(\gamma,u)\mapsto r\gamma+\angbra{\xi,u}$. We also use $\angbra{\bullet,\bullet}$ to denote the resulting pairing $(\R_{\geq0}\times N_{\R})\times(\Gamma\times \Mon)\to\R$. 
Thus, for any $\frakS\subseteq\Gamma\times\Mon$ we can consider $\frakS^\vee=\{ w\in\R_{\geq0}\times N_{\R}\;:\;\angbra{ w,\fraku}\leq0\text{ for all }\fraku\in\frakS\}.$

\begin{definition}
We denote by $\mathcal{P}_\bullet = (\mathcal{P}_0 \leq \mathcal{P}_1 \leq \cdots \leq \mathcal{P}_k)$ a 
\textit{flag of polyhedra}, where $\dim \mathcal{P}_i = i$. We assume that all polyhedra are strongly convex, non-empty and live in $N_\R$.

Similarly, we let $\calC_\bullet=(\calC_0\leq\calC_1\leq\cdots\leq\calC_{k})$ denote a \emph{flag of cones} in $\R_{\geq_0}\times N_{\R}$ with $\dim\calC_{i}=i+1$. We say that $\calC_\bullet$ is \emph{simplicial} if each $\calC_i$ is a simplicial cone.

\end{definition}

Given a flag $\calP_{\bullet}$ of polyhedra, the corresponding flag of cones is $c(\calP_\bullet)=(c(\calP_0)\leq c(\calP_1)\leq\cdots\leq c(\calP_k))$. The flags of cones $\calC_\bullet$ which arise this way are exactly those for which $\calC_0$ is not contained in $\{0\}\times N_{\R}$. For such flags $\calC_\bullet$, we write $\calC_\bullet|_{\{1\}\times N_{\R}}$ for the corresponding flag of polyhedra.

\newcommand{\matt}{C}
Consider the map taking a matrix $\matt$ whose first column is lexicographically at least 0 to the prime it defines on $\base[\Mon]$. Here, because we have not fixed an identification $\Mon\cong\Z^n$, $\matt$ is of the form $\matt=\begin{pmatrix}r_0&\xi_0\\r_1&\xi_1\\\vdots&\vdots\\r_k&\xi_k\end{pmatrix}$ where $r_0,\ldots,r_k\in\R$ and $\xi_0,\ldots,\xi_k\in N_{\R}$. We now show how this map factors through the set of simplicial flags of cones. Moreover, we show that if we restrict to matrices giving primes in $\ContBase{\base}\base[\Mon]$, the map factors through the corresponding set of flags of polyhedra.

Given a matrix $\matt$ as above, we can use downward gaussian elimination as in Lemma~\ref{lemma:rref} and removal of zero rows to get all of the entries in the first column non-negative and to get the rows linearly independent without changing the prime that $\matt$ defines. Then we let $\calC_\bullet(\matt)$ be the simplicial flag of cones defined by letting $\calC_i(\matt)$ be the cone generated by $(r_0,\xi_0),\ldots,(r_i,\xi_i)$ for each $0\leq i\leq k$.

For any simplicial flag of cones $\calC_\bullet$ in $\R_{\geq0}\times N_{\R}$, we get a prime congruence on $\base[\Mon]$ as follows.

Let $\calc_{-1}=\{0\}$ and pick points $ w_i \in \calC_i\sdrop\calc_{i-1}$ for $0 \leq i \leq k$. Since $ w_i\in \R_{\geq0}\times N_{\R}$, we can think of each $ w_i$ as a homomorphism $ w_i:\Gamma\times M\to\R$. Thus, we can consider the group homomorphism
$$\mbo{ w} = ( w_0, \ldots,  w_k):\Gamma \times \Mon \rightarrow \R^{k+1},$$ 
which preserves the order on $\Gamma$ when we give $\R^{k+1}$ the lexicographic order. Therefore, $\mbo{ w}$ gives rise to a semiring homomorphism
$$\varphi_{\mbo{ w}}:\base[\Mon] \rightarrow \RnLexSemifield{k+1}, \quad a\chi^u \mapsto (w_0, \ldots, w_k)(\log a, u).$$
We let $P_{\mbo{ w}}=\ker\ph_{\mbo{ w}}$. That is, $P_{\mbo{ w}}$ is the prime on $\base[\Mon]$ defined by the matrix 
$\begin{pmatrix} w_0\\\vdots\\ w_k\end{pmatrix}$.
By Proposition~\ref{coro: matrix_form}, $P_{\mbo{ w}}$ is in $\ContBase{\base}\base[\Mon]$ if and only if the (1,1) entry of this matrix is positive. Since $ w_0\in \calc_0\sdrop\{0\}\subseteq \R_{\geq0}\times N_{\R}$, this happens exactly when $\calc_0$ is not contained in $\{0\}\times N_{\R}$, i.e., when $\calC_\bullet=c(\calP_\bullet)$ for some flag $\calP_\bullet$ of polyhedra.

In order for this construction to give a prime defined by $\calC_\bullet$, we must show that $P_{\mbo{ w}}$ is independent of the choice of $\mbo{ w}$. Set $ v_0 =  w_0$. For $1\leq i\leq k$, $\calc_i\sdrop\calc_{i-1}$ contains a unique ray of $\calc_i$. Fix a generator of this ray and call it $v_i$. Consider the vector $\mbo{ v} = ( v_0, \ldots,  v_k)$. Analogously to $\varphi_{\mbo{ w}}$ and $P_{\mbo{ w}}$, we can define $\varphi_{\mbo{ v}}$ and $P_{\mbo{ v}}$. The following proposition gives us that $P_{\mbo{ v}}$ and $P_{\mbo{ w}}$ are the same.

\begin{proposition}\label{prop: autom-preserves-primes}
There is an automorphism $\psi$ of $\RnLexSemifield{k+1}$ given by a linear transformation on $\R^{k+1}$ such that $\varphi_{\mbo{ w}} = \psi \circ \varphi_{\mbo{ v}}$.
\end{proposition}

\begin{proof}
We proceed by induction on $k$. The case $k =0$ is trivial. Assume the statement is true for some $k = r$; we want to show that it is true for $k=r+1$.
Let $\tilde{\varphi}_{\mbo{ w}}$, $\tilde{\varphi}_{\mbo{ v}}$ and $\tilde{\psi}$ be the maps we obtain for the flag $(\mathcal{P}_0 \leq \mathcal{P}_1 \leq \cdots \leq \mathcal{P}_r)$. Write 
$$ w_{r+1} = \sum_{i=0}^{r+1} \alpha_i  v_i, \text{ where } \alpha_i \in \R_{\geq 0} \text{ and } \alpha_{r+1} >0.$$
Thus, we have
$$\begin{pmatrix} w_0 \\  w_1 \\ \vdots \\  w_r \\  w_{r+1} \end{pmatrix} 
=
\left(\begin{array}{@{}c|c@{}}
\tilde{\psi}
  &     \begin{matrix}
  0  \\
  0  \\
  \vdots \\
  0
  \end{matrix} \\
\hline
   \begin{matrix}
 \alpha_0 & \alpha_1 & \cdots & \alpha_r 
  \end{matrix} &
\alpha_{r+1}
\end{array}\right)
 \begin{pmatrix} v_0 \\  v_1 \\ \vdots \\  v_r \\  v_{r+1} \end{pmatrix},$$
where we call the matrix in this product $\psi$. Note that it is invertible since $\tilde{\psi}$ is and $\alpha_{r+1} >0$. Moreover, $\psi$ is order preserving, i.e., $\mbo{h} > \mbo{h}'$ implies that $\psi \mbo{h} > \psi \mbo{h'}$. We can show this by induction; the base case is trivial. If $\tilde{\psi}$ is order preserving and then $\psi$ is, since $ v_{r+1} >  v_{r+1}'$ if and only if $\alpha_{r+1}  v_{r+1}>\alpha_{r+1}  v_{r+1}'$ and $\alpha_{r+1}>0$.
\end{proof}

Now consider $\mbo{ w}'=( w'_0,\ldots, w'_k)$ with $ w'_i\in\calc_i\sdrop\calc_{i-1}$. In the same way as we got $\mbo{ v}$ from $\mbo{ w}$, we now get $\mbo{ v}'$. Since $ v_i$ and $ v'_i$ are generators of the same ray, they are positive multiples of each other. Thus, $P_{\mbo{ w}'}=P_{\mbo{ v}'}=P_{\mbo{ v}}=P_{\mbo{ w}}$. In light of this we can make the following definition. 

\begin{defi}\label{def:PrimeOfAFlag}
Let $\calc_\bullet$ be a simplicial flag of cones. The \emph{prime congruence $P_{\calc_\bullet}$ defined by $\calC_\bullet$} is $P_{\mbo{ w}}$ for any choice of $\mbo{ w}=( w_0,\ldots, w_k)$ with $ w_i\in\calc_i\sdrop\calc_{i-1}$. If $\calc_\bullet=c(\calp_\bullet)$, then $P_{\calp_\bullet}=P_{c(\calP_\bullet)}$ is the \emph{prime congruence defined by $\calp_\bullet$}.
\end{defi}

If $\calc_\bullet=\calc_\bullet(\matt)$ for some matrix $\matt$, then we can chose $ w_i$ to be the $(i+1)^\text{st}$ row of $\matt$, so $\matt$ is a defining matrix for $P_{\calc_\bullet}$. In particular, for every prime congruence $P$ on $\base[\Mon]$, there is a simplicial flag $\calc_\bullet$ of cones such that $P=P_{\calc_\bullet}$. If $P\in\ContBase{\base}\base[\Mon]$, then we know that $\calc_\bullet=c(\calp_\bullet)$ for some flag $\calp_\bullet$ of polyhedra, so $P=P_{\calp_\bullet}$.

\subsection{Equivalence of flags.}

Recall that a \textit{$\Gamma$-rational polyhedron} is a polyhedron in $N_{\R}$ which can be written as $\{x\in N_{\R}\ :\ \angbra{x,u_i}\leq\gamma_i\text{ for }i=1,\ldots,q\}$ for some $u_1,\ldots,u_q\in \Mon$ and $\gamma_1,\ldots,\gamma_q\in\Gamma$. A \textit{$\Gamma$-rational polyhedral set} is a finite union of $\Gamma$-rational polyhedra. 

A cone in $\R_{\geq0}\times N_{\R}$ is \emph{$\Gamma$-admissible} if it can be written as 
$$\{(r,x)\in \R_{\geq0}\times N_{\R}\ :\ r\gamma_i+\angbra{x,u_i}\leq0\text{ for }i=1,\ldots,q\}$$
for some $u_1,\ldots,u_q\in \Mon$ and $\gamma_1,\ldots,\gamma_q\in\Gamma$. 
A polyhedron $\calQ$ is $\Gamma$-rational if and only if $c(\calQ)$ is $\Gamma$-admissible. A cone contained in $\{0\}\times N_{\R}$ is $\Gamma$-admissible if and only if it is rational. We call a finite union of $\Gamma$-admissible cones a \emph{$\Gamma$-admissible \conoidSet}. Note that the set of $\Gamma$-rational polyhedral sets in $N_{\R}$ and the set of $\Gamma$-admissible \conoidSets\ each form a lattice when ordered by inclusion. In both of these lattices, meet and join are given by intersection and union, respectively.

\begin{definition}\label{def: gamma-rational-nbhd}

Let $\mathcal{P}_\bullet$ be a flag of polyhedra and let $\calc_\bullet$ be a flag of cones. A $\Gamma$\textit{-rational neighborhood} of $\mathcal{P}_\bullet$ is a $\Gamma$-rational polyhedron $\mathcal{Q}$ which meets the relative interior of each $\mathcal{P}_i$. A \emph{$\Gamma$-admissible neighborhood} of $\calc_\bullet$ is a $\Gamma$-admissible cone $\calD$ which meets the relative interior of each $\calc_i$.
\end{definition}

\begin{definition}\label{def: local-equiv}
We say that two flags $\mathcal{P}_\bullet$ and $\mathcal{P}'_\bullet$ of polyhedra are ($\Gamma$-)\textit{locally equivalent} if the $\Gamma$-rational neighborhoods of $\mathcal{P}_\bullet$ are exactly the $\Gamma$-rational neighborhoods of $\mathcal{P}'_\bullet$. Similarly, we say that a two flags $\calc_\bullet$ and $\calc'_\bullet$ are ($\Gamma$-)\emph{locally equivalent} if they have the same $\Gamma$-admissible neighborhoods.
\end{definition}

Note that $\calQ$ is a $\Gamma$-rational neighborhood of $\calP_\bullet$ if and only if $c(\calQ)$ is a $\Gamma$-admissible neighborhood of $c(\calP_\bullet)$. Thus $\calP_\bullet$ and $\calP'_\bullet$ are locally equivalent if and only if $c(\calP_\bullet)$ and $c(\calP'_\bullet)$ are.

The following lemma provides an alternate characterization of neighborhoods of flags. To prove it we will use the fact that a point $x$ is in the relative interior of a cone $\calC$ of dimension $k$ if and only if $x$ can be written as a positive linear combination of $k$ linearly independent points in $\calC$. See \cite[Lemmas 2.8 and 2.9]{Zieg95} for the proof of the corresponding fact about polytopes. The proof for cones is directly analogous.

\begin{lemma}

A $\Gamma$-admissible cone $\calD$ is a neighborhood of $\calc_\bullet$ if and only if $\calD$ meets $\calc_i\sdrop\calc_{i-1}$ for $0\leq i\leq k$. In particular,
a $\Gamma$-rational polyhedron $\mathcal{Q}$ is a neighborhood of $\mathcal{P}_\bullet = (\mathcal{P}_0 \leq \mathcal{P}_1 \leq \cdots \leq \mathcal{P}_k)$ if and only if $\mathcal{P}_0 \in \mathcal{Q}$ and $\mathcal{Q}$ meets $\mathcal{P}_i\setminus\mathcal{P}_{i-1}$, for $1 \leq i \leq k$.
\end{lemma}

\begin{proof}

One direction follows from the fact that $\relint\calc_i\subseteq\calc_i\sdrop\calc_{i-1}$. 
For the other direction, we will proceed by induction on $k$. The case $k=0$ is trivial because $\relint \calc_0=\calc_0\sdrop\calc_{-1}$. Assume the statement holds for $k-1$. By the inductive hypothesis, $\cald$ meets the relative interior of $\calc_{i}$ for $1\leq i<k$, so we only need to show that $\cald$ meets the relative interior of $\calc_k$. Pick a point $x \in \cald \cap \relint \calc_{k-1}$. So $x$ is a positive linear combination of $k$ linearly independent points in $\calc_{k-1}$. Because $\cald$ meets $\calc_k\sdrop\calc_{k-1}$, we can pick a point $y \in \cald \cap \calc_k\sdrop\calc_{k-1}$. Then $x+y$ is a positive linear combination of $k+1$ linearly independent points in $\calc_k$, so $x+y\in\relint\calc_k$. Since $x,y\in\cald$ and $\cald$ is a cone, we get $x+y\in\cald\cap\relint\calc_k$.

The claim about polyhedra follows by considering $\calc_\bullet=c(\calp_\bullet)$ and $\calD=c(\calq)$ and noting that, because $\calc_0$ is a ray, $\calc_0\sdrop\calc_{-1}=\calc_0\sdrop\{0\}$ meets the cone $\calD$ if and only if $\calc_0\sdrop\{0\}$ is contained in $\cald$.
\end{proof}

We now work towards showing that every flag of cones is locally equivalent to one that is simplicial. To do this, we will need a technical lemma.

\newcommand{\coneSigma}{\calc}
\newcommand{\coneTau}{\calc'}
\newcommand{\coneF}{\calb}
\newcommand{\coneG}{\calb'}
\newcommand{\conesigma}{\coneSigma}
\newcommand{\conetau}{\coneTau}
\newcommand{\conef}{\coneF}
\newcommand{\coneg}{\coneG}
\begin{lemma}\label{lemma:InductiveStepInSimplicialReduction}
Let $\coneSigma$ be a cone and let $\coneF$ be a facet of $\coneSigma$. Let $\coneG$ be a subcone of $\conef$ with such that $\dim \coneg=\dim\conef$ and let $\conetau$ be a subcone of $\conesigma$ with $\dim\conetau=\dim\conesigma$ and $\coneg$ a face of $\conetau$. Suppose that $v\in\relint\coneg$ and $w\in\relint\conesigma$. Then there is a real number $\epsilon>0$ such that $v+\epsilon w\in\conetau$.
\end{lemma}\begin{proof}
Since $\conef$ is a face of $\conesigma$, there is some $u\in\Mon_{\R}$ such that  such that $\conesigma\subseteq u^\vee$ and $\conef=\conesigma\cap u^\perp$. In particular, $v\in u^\perp$ and $w\in u^\vee$. Since $\coneg=\conetau\cap u^\perp$ is a facet of $\conetau$ and $v\in\relint\coneg$, the star of $\conetau$ at $v$ is $u^\vee$. That is, $u^\vee=\R_{\geq0}(\conetau-v)$. Since $w\in u^\vee$ is nonzero, this means that we can write $w=r(z-v)$ for some $r>0$ and $z\in\conetau$. So, letting $\epsilon=\frac{1}{r}$, we have $v+\epsilon w=v+(z-v)=z\in\conetau$.
\end{proof}

This lemma will help us with an inductive argument. The following definition will also facilitate this argument.

\begin{definition}
Let $\calc_\bullet = (\calc_0 \leq \calc_1 \leq \cdots \leq \calc_k)$ be a flag of cones. 
For any $j\leq k$, the \textit{truncation} of $\calc_\bullet$ at $j$ is $\calc_\bullet^{(j)} = (\calc_0 \leq \calc_1 \leq \cdots \leq \calc_j)$. 
\end{definition}

\begin{prop}\label{prop:EveryFlagLocEquivToSimplicial}
Any flag $\calc_\bullet$ of cones is locally equivalent to a simplicial flag $\calc_\bullet'$ of cones. Any flag $\calp_\bullet$ of polyhedra is locally equivalent to a flag $\calp_\bullet'$ of cones such that $c(\calp_\bullet')$ is simplicial.
\end{prop}\begin{proof}
Consider any flag of cones $\calc_\bullet=(\calc_0\leq\cdots\leq\calc_k)$ and, for each $0\leq i\leq k$ pick a ray $\rho_i$ of $\calc_i$ not contained in $\calc_{i-1}$. Define a simplicial flag $\calc_\bullet'$ of cones by setting $\calc_i'=\dsum_{j=0}^i\rho_j$. Note that $\calc_i'$ is a subcone of $\calc_i$ with $\dim\calc_i'=i+1=\dim\calc_i$ and $\calc_i'\sdrop\calc_{i-1}'\subseteq\calc_i\sdrop\calc_{i-1}$. In particular, any $\Gamma$-admissible neighborhood of $\calc_\bullet'$ is a $\Gamma$-admissible neighborhood of $\calc_\bullet$.

We now prove that any $\Gamma$-admissible neighborhood of $\calc_\bullet$ is a $\Gamma$-admissible neighborhood of $\calc_\bullet'$ by induction on $k$. For the base of $k=0$, note that $\calc_0'=\calc_0$, so we are done. Now assume that $k\geq1$ and the result is true for $k-1$. Suppose that $\cald$ is a $\Gamma$-admissible neighborhood of $\calc_\bullet$. Then $\cald$ is a neighborhood of $\calc_\bullet^{(k-1)}$ so, by the inductive hypothesis, $\cald$ is a neighborhood of $\calc_\bullet'^{(k-1)}$. Thus $\cald$ meets the relative interior of $\calc_i'$ for $i<k$, and it suffices to show that $\cald$ meets $\calc_k'\sdrop\calc_{k-1}'$. We can pick $v\in(\relint\calc_{k-1}')\cap\cald$ and $w\in(\relint\calc_k)\cap\cald$. By Lemma~\ref{lemma:InductiveStepInSimplicialReduction}, there is an $\epsilon>0$ such that $v+\epsilon w\in\calc_k'$. Since  $\epsilon w\in\relint\calc_k$ and $\calc_{k-1}$ is a proper face of $\calc_k$, $v\in\calc_{k-1}$ gives us that $v+\epsilon w\notin\calc_{k-1}\supseteq\calc_{k-1}'$. Thus, $v+\epsilon w\in\calc_k'\sdrop\calc_{k-1}'$. Since $v,w\in\cald$ and $\cald$ is a cone, $v+\epsilon w\in\cald$.

Given any flag $\calp_\bullet$ of polyhedra, consider the corresponding flag $\calc_\bullet=c(\calp_\bullet)$ of cones and let $\calc_\bullet'$ be as above. Since $\calc_0'=\calc_0$ is not contained in $\{0\}\times N_{\R}$, there is a flag $\calp_\bullet'$ of polyhedra such that $\calc_\bullet'=c(\calp_\bullet')$. Since $c(\calp_\bullet)$ and $c(\calp_\bullet')$ are locally equivalent, so are $\calp_\bullet$ and $\calp_\bullet'$.
\end{proof}

\subsection{The filter of a prime congruence.}\label{subsec:FilterOfAPrimeCong}

Now we introduce the next key player in this paper. For 
finitely many Laurent polynomials
$f_0, \ldots, f_n \in \base[\Mon]$ we define the \textit{rational set} $R(f_0, \ldots, f_n )$ to be
$$R(f_0, \ldots, f_n ) = \{ x \in N_\R\,:\, f_0(x) \geq f_i(x),\text{ for } 1 \leq i \leq n \}.$$
Here, when $f=\sum_{u\in\Mon}f_u\chi^u\in\base[\Mon]$ and $x\in N_{\R}$, we let $f(x)=\max_{u\in\Mon}\left(\log(f_u)+_{\R}\angbra{x,u}\right)$.

For any $f\in\base[\Mon]$, we can form the homogenized function $\widetilde{f}$ given by 
$$\widetilde{f}(r,x)=\max_{u\in\Mon}\left(r\log(f_u)+_{\R}\angbra{x,u}\right)$$
for any $(r,x)\in\R_{\geq0}\times N_{\R}$. We let 
$$\widetilde{R}(f_0,f_1,\ldots,f_n)=\left\{ w\in\R_{\geq0}\times N_{\R}\,:\,\wt{f}_0( w)\geq\wt{f}_i( w)\text{ for }1\leq i\leq n\right\}.$$ 
Note that $R(f_0,f_1,\ldots,f_n)$ is the restriction of $\wt{R}(f_0,f_1,\ldots,f_n)$ to $\{1\}\times N_{\R}$.

\begin{lemma}\label{lemma:BasicPropertiesOfRatSets}

Say $f_0,\ldots,f_n\in\base[\Mon]$ and write $f_i=\sum_{u\in\Mon}f_{i,u}\chi^u$. Then
\begin{enumerate}
\item\label{item:RationalSetAsIntersectionOfBinaryRatSets} $\wt{R}(f_0,f_1, \ldots, f_n ) = \bigcap_{i=1}^n \wt{R}(f_0, f_i)$,
\item\label{item:RationalSetsAndTerms} $\wt{R}(f_0, f_i) = \bigcap_{u\in M} \wt{R}(f_0, f_{i, u} \chi^u) = \bigcup_{u\in M} \wt{R}(f_{0, u} \chi^u, f_i)$, and
\item\label{item:RationalSetPickingOneTerm} for any $\mu\in\Mon$, $ \wt{R}(f_0,f_1, \ldots, f_n ) \supseteq \wt{R}(f_{0, \mu} \chi^\mu, f_1, \ldots, f_n )$.
\end{enumerate}
The same applies if we replace $\wt{R}$ with $R$.
\end{lemma}
\begin{proof}
(\ref{item:RationalSetAsIntersectionOfBinaryRatSets}) follows from the definition. (\ref{item:RationalSetsAndTerms}) and (\ref{item:RationalSetPickingOneTerm}) follow from the fact that, in the evaluation of 
$\wt{f}( w)$,
we take the maximum over the terms of $f$. The final claim follows by restriction to $\{1\}\times N_{\R}$.
\end{proof}

\begin{defi}

For any prime congruence $P$ on $\base[\Mon]$, the \emph{prime filter $\wt{\calf}_P$ that $P$ defines on the \ConoidLarry}\ is the collection of $\Gamma$-admissible \conoidSets\ $\wt{U}$ for which there are $f_0,f_1, \ldots, f_n \in \base[\Mon]$ such that $\wt{R}(f_0,f_1, \ldots, f_n ) \subseteq \wt{U}$ and 
$|f_0|_P \geq |f_i|_P$ for $1\leq i\leq n$. 
If $P \in \ContBase{\base}{\base[\Mon]}$, then the \emph{prime filter $\calf_P$ that $P$ defines on the \BigLarry}\ is the collection of $\Gamma$-rational polyhedral sets $U$ in $N_\R$ for which there are  $f_0,f_1, \ldots, f_n \in \base[\Mon]$ such that $R(f_0,f_1,\ldots,f_n)\subseteq U$ and $|f_0|_P \geq |f_i|_P$ for $1\leq i\leq n$. In this case, a $\Gamma$-rational polyhedral set $U$ is in $\calf_P$ if and only if $c(U)$ is in $\wt{\calf}_P$. 
\end{defi}

We will justify these names later in this section.
We begin our study of $\wt{\calf}_P$ and $\calF_P$ by observing that, in the above definitions, it is enough to consider monomials.

\begin{lemma}

Let $P$ be a prime congruence on $\base[\Mon]$ and $\wt{U}$ a $\Gamma$-admissible \conoidSet. Then $\wt{U}\in\wt{\calF}_P$ if and only if there are terms $a_1\chi^{u_1},\ldots,a_n\chi^{u_n}\in \base[\Mon]$ such that $\wt{R}(1_\base,a_1\chi^{u_1},\ldots,a_n\chi^{u_n})\subseteq \wt{U}$ and $1_{\kappa(P)}\geq|a_i\chi^{u_i}|_P$ for $1\leq i\leq n$.

If $P\in\ContBase{\base}\base[\Mon]$ and $U$ is a $\Gamma$-rational polyhedral set, then $U\in\calF_P$ if and only if there are terms $a_1\chi^{u_1},\ldots,a_n\chi^{u_n}\in \base[\Mon]$ such that $R(1_\base,a_1\chi^{u_1},\ldots,a_n\chi^{u_n})\subseteq U$ and $1_{\kappa(P)}\geq|a_i\chi^{u_i}|_P$ for $1\leq i\leq n$.
\end{lemma}\begin{proof}

The ``if'' direction is clear. 

For the other direction, suppose $\wt{U}\in\wt{\calF}_P$, i.e., there are $f_0,f_1,\ldots,f_n\in \base[\Mon]$ such that $\wt{R}(f_0,f_1,\ldots,f_n)\subseteq \wt{U}$ and $|f_0|_P\geq|f_i|_P$ for $1\leq i\leq n$. 
Let $a_0\chi^{u_0}$ be a term of $f_0$ such that the maximum in $|f_0|_P=\dsum_{u\in M} |f_{0,u}\chi^{u}|_P$ is attained at $|a_0\chi^{u_0}|_P$. Then we have that $|a_0\chi^{u_0}|_P\geq |f_i|_P$ for $1\leq i\leq n$ and, by Lemma~\ref{lemma:BasicPropertiesOfRatSets}~ (\ref{item:RationalSetPickingOneTerm}), $\wt{R}(a_0\chi^{u_0},f_1,\ldots,f_n)\subseteq \wt{R}(f_0,f_1,\ldots,f_n)\subseteq \wt{U}$. Thus, we may assume without loss of generality that $f_0=a_0\chi^{u_0}$. Also, $a_0\chi^{u_0}$ is a unit in $\base[\Mon]$ and $|a_0\chi^{u_0}|_P\geq |f_i|_P$ gives us that $1_{\kappa(P)}\geq|(a_0\chi^{u_0})^{-1}f_i|_P$. Since also $\wt{R}(a_0\chi^{u_0},f_1,\ldots,f_n)=\wt{R}(1_\base,(a_0\chi^{u_0})^{-1}f_1,\ldots,(a_0\chi^{u_0})^{-1}f_n)$, we may assume without loss of generality that $f_0=1_\base$.

By Lemma~\ref{lemma:BasicPropertiesOfRatSets}~ (\ref{item:RationalSetAsIntersectionOfBinaryRatSets}) and (\ref{item:RationalSetsAndTerms}), $\wt{R}(1_\base,f_1,\ldots,f_n)=\dcap_{i=1}^n \wt{R}(1_\base,f_i)=\dcap_{i=1}^n\dcap_{u\in\Mon}\wt{R}(1_\base,f_{i,u}\chi^u)$. So, if we relabel those $f_{i,u}\chi^u$ with $f_{i,u}\neq0_\base$ as $a_1\chi^{u_1},\ldots,a_n\chi^{u_n}$, then we have $\wt{R}(1_\base,f_1,\ldots,f_n)=\dcap_{i=1}^n \wt{R}(1_\base,a_i\chi^{u_i})=\wt{R}(1_\base,a_1\chi^{u_1},\ldots,a_n\chi^{u_n})$.

The proof of the statement for $\calf_P$ is analogous.
\end{proof}

Having reduced to considering the sets $R(1_\base,a_1\chi^{u_1},\ldots,a_n\chi^{u_n})$ and $\wt{R}(1_\base,a_1\chi^{u_1},\ldots,a_n\chi^{u_n})$, we now consider the geometric nature of these sets.

\begin{lemma}\label{lemma:RatSetsAreGammaRat/Ad}

For any terms $a_1\chi^{u_1},\ldots,a_n\chi^{u_n}\in\base[\Mon]$, $\wt{R}(1_\base,a_1\chi^{u_1},\ldots,a_n\chi^{u_n})$ is a $\Gamma$-admissible cone. If $P\in\ContBase{\base}\base[\Mon]$ and $1_{\kappa(P)}\geq|a_i\chi^{u_i}|_P$ for $1\leq i\leq n$ then $R(1_\base,a_1\chi^{u_1},\ldots,a_n\chi^{u_n})$ is a $\Gamma$-rational polyhedron. 
\end{lemma}\begin{proof}

Note that, for $1\leq i\leq n$, $\tilde{R}(1_\base,a_i\chi^{u_i})=\{(r,x)\in\R_{\geq0}\times N_{\R}\;:\;0_{\R}\geq r\log(a_i)+_{\R}\angbra{x,u}\}$ is a $\Gamma$-admissible half-space. So $\wt{R}(1_\base,a_1\chi^{u_1},\ldots,a_n\chi^{u_n})=\dcap_{i=1}^n \wt{R}(1_\base,a_i\chi^{u_i})$ is a $\Gamma$-admissible cone.

Now 
say 
$P\in\ContBase{\base}\base[\Mon]$. Since $R(1_\base,a_1\chi^{u_1},\ldots,a_n\chi^{u_n})$ can be written as the restriction of $\wt{R}(1_\base,a_1\chi^{u_1},\ldots,a_n\chi^{u_n})$ to $\{1\}\times N_{\R}$, it remains only to show that $R(1_\base,a_1\chi^{u_1},\ldots,a_n\chi^{u_n})$ is nonempty.
By Proposition~\ref{coro: matrix_form}, we can pick a defining matrix for $P$ of the form $\begin{pmatrix}1 & \xi_0\\0 &\xi_1 \\ \vdots & \vdots  \\ 0 & \xi_k \end{pmatrix}$ with $\xi_0,\ldots,\xi_k$ in $N_{\R}$. Since $1_{\kappa(P)}\geq|a_i\chi^{u_i}|_P$, we have that $0\geq\log(a_i)+_{\R}\angbra{\xi_0,u_i}$, i.e., $\xi_0\in R(1_\base,a_i\chi^{u_i})$. Therefore, $\xi_0\in\dcap_{i=1}^n R(1_\base,a_i\chi^{u_i})= R(1_\base,a_1\chi^{u_1},\ldots,a_n\chi^{u_n}).$

\end{proof}

At the end of the proof of the previous lemma, we took advantage of the fact that we could evaluate whether $1_{\kappa(P)}\geq|a\chi^{u}|_P$ by picking a defining matrix for $P$ and using the lexicographic order. We now spell out the definition of the lexicographic order as it applies here and work towards translating it into geometric conditions.
Let 
$$\matt=\begin{pmatrix} v_0\\ v_1 \\ \vdots  \\  v_k\end{pmatrix}$$ be a defining matrix for a prime congruence $P$ on $\base[\Mon]$. Let $\mathfrak{u}=\begin{pmatrix}
\log(a)\\u
\end{pmatrix}$ be the exponent vector of $a\chi^u$. Then 
$1_\base \geq |a\chi^{u}|_P$ is equivalent to 
$$\begin{pmatrix} \angbra{ v_0,\mathfrak{u}}\\ \angbra{ v_1,\mathfrak{u}}\\  \vdots  \\  \angbra{ v_k,\mathfrak{u}} \end{pmatrix} \leq_{lex} \begin{pmatrix} 0 \\ 0 \\ \vdots \\ 0\end{pmatrix},$$
which happens whenever
\begin{align}\label{eq: conditions-rational-sets}
    \angbra{ v_0,\mathfrak{u}} &< 0 \text{, or } \nonumber\\
    \angbra{ v_0,\mathfrak{u}} &= 0 \text{ and  } \angbra{ v_1,\mathfrak{u}} < 0\text{, or } \nonumber\\
    &\ \ \  \vdots \\
    \angbra{ v_0,\mathfrak{u}} &= 0 , \ldots, \angbra{ v_{k-2},\mathfrak{u}}= 0 \text{, and  }\angbra{ v_{k-1},\mathfrak{u}}<0 \text{, or } \nonumber\\
    \angbra{ v_0,\mathfrak{u}}&= 0 , \ldots, \angbra{ v_{k-1},\mathfrak{u}}= 0 \text{, and  }\angbra{ v_k,\mathfrak{u}}\leq 0. \nonumber
\end{align}
Equivalently,
\begin{align}\label{eq: conditions-rational-eqiv}
    \angbra{ v_0,\mathfrak{u}} &< 0 \text{, or } \nonumber\\
    \angbra{ v_0,\mathfrak{u}} &\leq 0 \text{ and  } \angbra{ v_1,\mathfrak{u}} < 0\text{, or } \nonumber\\
    &\ \ \  \vdots \\
    \angbra{ v_0,\mathfrak{u}} &\leq 0 , \ldots, \angbra{ v_{k-2},\mathfrak{u}}\leq 0 \text{, and  }\angbra{ v_{k-1},\mathfrak{u}}<0 \text{, or } \nonumber\\
    \angbra{ v_0,\mathfrak{u}}&\leq 0 , \ldots, \angbra{ v_{k-1},\mathfrak{u}}\leq 0 \text{, and  }\angbra{ v_k,\mathfrak{u}}\leq 0. \nonumber
\end{align}


For convenience, we label the conditions in (\ref{eq: conditions-rational-sets}) as (\ref{eq: conditions-rational-sets}.0), (\ref{eq: conditions-rational-sets}.1), $\ldots$, (\ref{eq: conditions-rational-sets}.$k-1$), and (\ref{eq: conditions-rational-sets}.$k$). Similarly, we label the conditions in (\ref{eq: conditions-rational-eqiv}) as (\ref{eq: conditions-rational-eqiv}.0), (\ref{eq: conditions-rational-eqiv}.1), $\ldots$, (\ref{eq: conditions-rational-eqiv}.$k-1$), and (\ref{eq: conditions-rational-eqiv}.$k$).

\begin{remark}

Suppose $\matt'$ is obtained from $\matt$ by downwards gaussian elimination. The proof of Lemma~\ref{lemma:rref} not only shows that $\matt\mathfrak{u}\leq_{lex}\mbo{0}$ if and only if $\matt'\mathfrak{u}\leq_{lex}\mbo{0}$, but that, for $0\leq i\leq k$, (\ref{eq: conditions-rational-sets}.$i$) is satisfied for $\matt\mathfrak{u}\leq_{lex}\mbo{0}$ if and only if it is satisfied for $\matt'\mathfrak{u}\leq_{lex}\mbo{0}$. Since (\ref{eq: conditions-rational-eqiv}.$i$) is satisfied exactly if (\ref{eq: conditions-rational-sets}.$j$) is satisfied for some $j\leq i$, we also get that (\ref{eq: conditions-rational-eqiv}.$i$) is satisfied for $\matt\mathfrak{u}\leq_{lex}\mbo{0}$ if and only if it is satisfied for $\matt'\mathfrak{u}\leq_{lex}\mbo{0}$.
\end{remark}

The following proposition gives a geometric interpretation to the above conditions. In order to state it, we define $\wt{R}^\circ(f,g)=\{ w\in\R_{\geq0}\times N_{\R}:\wt{f}( w)>\wt{g}( w)\}$. When we are considering a half-space $\wt{R}(1_\base,a\chi^u)$, $\wt{R}^\circ(1_\base,a\chi^u)$ is the corresponding strict half-space.

\begin{prop}\label{prop:GeomEvalOfPNorm}

Let $\calc_\bullet$ be a simplicial flag of cones, and let $P=P_{\calc_\bullet}$. For any term $a\chi^u\in\base[\Mon]$, $1_{\kappa(P)}\geq|a\chi^u|_P$ if and only if 
\begin{align}\label{eq: conditions-rational-sets-geom}
    &\relint\calc_0\subseteq\wt{R}^\circ(1_\base, a \chi^{u}) \text{, or } \nonumber\\
    &\relint\calc_1\subseteq\wt{R}^\circ(1_\base, a \chi^{u})\text{, or}\nonumber\\
    &\qquad\qquad\quad \vdots\\
    &\relint\calc_{k-1}\subseteq\wt{R}^\circ(1_\base, a \chi^{u})\text{, or }\nonumber\\
    &\calc_k\subseteq\wt{R}(1_\base, a \chi^{u}).\nonumber
\end{align}
\end{prop}

\begin{remark}\label{rem:ResultsOfConditionsInProp:GeomEvalOfPNorm}
As before, we label the conditions in (\ref{eq: conditions-rational-sets-geom}) as (\ref{eq: conditions-rational-sets-geom}.0), (\ref{eq: conditions-rational-sets-geom}.1), $\ldots$, (\ref{eq: conditions-rational-sets-geom}.$k-1$), and (\ref{eq: conditions-rational-sets-geom}.$k$). Note that, if (\ref{eq: conditions-rational-sets-geom}.$i$) is satisfied, then, by taking closures in $\R_{\geq0}\times N_{\R}$, we get that $\calc_i\subseteq\wt{R}(1_\base,a\chi^u)$ and so also $\calc_j\subseteq\wt{R}(1_\base,a\chi^u)$ for all $j\leq i$. In particular, if $1_{\kappa(P)}\geq|a\chi^u|_P$ then $\calc_0\subseteq\wt{R}(1_\base,a\chi^u)$.
\end{remark}
\begin{proof}[Proof of Proposition~\ref{prop:GeomEvalOfPNorm}]
Fix a matrix $\matt=\begin{pmatrix}
 v_0\\ v_1\\\vdots\\ v_k
\end{pmatrix}$ such that $\calc_\bullet=\calc_\bullet(\matt)$.
 Let $a\chi^u$ be any term of $\base[\Mon]$ and let $\mathfrak{u}=\begin{pmatrix}
\log(a)\\u
\end{pmatrix}$ be the exponent vector of $a\chi^u$.
\newcommand{\acoeff}{\epsilon}
Suppose that $1_{\kappa(P)}\geq|a\chi^u|_P$, and, specifically, (\ref{eq: conditions-rational-eqiv}.$i$) is satisfied. If $i=k$ then $ v_0,\ldots, v_k$ are contained in the half-space $\wt{R}(1_\base,a\chi^u)$, so the cone $\calc_k$ that they generate is also contained in $\wt{R}(1_\base,a\chi^u)$, i.e., (\ref{eq: conditions-rational-sets-geom}.$k$) is satisfied. So now suppose $i<k$, and consider any $ w\in\relint\calc_i$. Since $\calc_i$ is the simplicial cone generated by $ v_0, v_1,\ldots, v_i$, we can write $ w=\dsum_{j=0}^i \acoeff_j v_j$ with $\acoeff_j>0$. Since we have $\angbra{ v_j,\mathfrak{u}}\leq 0$ for $j<i$ and $\angbra{ v_i,\mathfrak{u}}<0$, we have
$$\angbra{ w,\mathfrak{u}}=\dsum_{j=0}^i\acoeff_j\angbra{ v_j,\mathfrak{u}}\leq\acoeff_i\angbra{ v_i,\mathfrak{u}}<0,$$
so $ w\in\wt{R}^\circ(1_\base,a\chi^u)$. Thus (\ref{eq: conditions-rational-sets-geom}.$i$) is satisfied.

Now suppose that (\ref{eq: conditions-rational-sets-geom}.$i$) is satisfied for some $0\leq i\leq k$. If $i=k$ then $ v_0,\ldots, v_k\in\calc_k\subseteq\wt{R}(1_\base,a\chi^u)$, so $\angbra{ v_j,\mathfrak{u}}\leq0$ for all $0\leq j\leq k$, i.e., (\ref{eq: conditions-rational-eqiv}.$k$) is satisfied. So now suppose $i<k$, giving us $\relint\calc_i\subseteq\wt{R}^\circ(1_\base,a\chi^u)$ and thus $\calc_i\subseteq\wt{R}(1_\base,a\chi^u)$. In particular, for all $j\leq i$, $ v_j\in\calc_i\subseteq\wt{R}(1_\base,a\chi^u)$ gives us that $\angbra{ v_j,\mathfrak{u}}\leq 0$. Pick a point $ w\in\relint\calc_i$, so $\angbra{ w,\mathfrak{u}}<0$. Since $ w$ is in the relative interior of the simplicial cone generated by $ v_0,\ldots, v_i$, we can write $ w=\dsum_{j=0}^i\acoeff_j v_j$ with $\acoeff_j>0$. Since
$$0>\angbra{ w,\mathfrak{u}}=\dsum_{j=0}^i\acoeff_j\angbra{ v_j,\mathfrak{u}},$$
we must have that $\angbra{ v_j,\mathfrak{u}}<0$ for some $j\leq i$. Then (\ref{eq: conditions-rational-eqiv}.$j$) is satisfied.
\end{proof}

With this background, we are now prepared to show that $\wtcalf_P$ and $\calf_P$ are prime filters.

\begin{theorem}\label{thm:TheFilterIsAPrimeFilter}
Let $P$ be a prime congruence on $\base[\Mon]$. Then $\wt{\calf}_P$ is a prime filter on the lattice of \conoidSets. If $P\in \ContBase{\base}{\base[\Mon]}$, then $\mathcal{F} = \mathcal{F}_P$ is a prime filter on the \BigLarry. 
\end{theorem}

\begin{proof}
We need to check that $\wt{\calf}=\wt{\calf}_P$ satisfies the conditions in the definition of prime filter on a lattice, namely, (1) $\wt{\calf}$ is not the whole lattice, (2) $\wt{\calf}$ is not empty, (3) if $\wt{U}_1,\wt{U}_2\in\wt{\calf}$ then $\wt{U}_1\cap\wt{U}_2\in\wt{\calf}$, (4) if $\wt{U}\in\wt{\calf}$ and $\wt{V}\supseteq\wt{U}$ is a $\Gamma$-admissible \conoidSet, then $\wt{V}\in\calf_P$, and 
(5) if $\wt{U}=\dcup_{i=1}^m\wt{U}_i$ is in $\wt{\calf}$ and each $\wt{U}_i$ is a $\Gamma$-admissible \conoidSet, then some $\wt{U}_{i_0}$ is in $\wt{\calf}$.

(1) Write $P=P_{\calc_\bullet}$ for some simplicial flag $\calc_\bullet$ of cones. If $\wt{U}\in\wt{\calf}$ then $\calc_0\subseteq\wt{U}$. So, for any $\Gamma$-admissible ray $\rho\neq\calc_0$, $\rho\notin\wt{\calf}$.

(2) We have $\R_{\geq0}\times N_{\R}=\wt{R}(1_\base,0_\base)$, so $\R_{\geq0}\times N_\R\in\wt{\calf}$.

(3) Say $\wt{U}_1,\wt{U}_2\in\wt{\calf}$, so $\wt{U}_1\supseteq\dcap_{j\in J_1}\wt{R}(1_\base,a_j\chi^{u_j})$ and $\wt{U}_2\supseteq\dcap_{j\in J_2}\wt{R}(1_\base,a_j\chi^{u_j})$ for some finite indexing sets $J_1$ and $J_2$. Then $\wt{U}_1\cap\wt{U}_2\supseteq\dcap_{j\in J_1\cup J_2}\wt{R}(1_\base,a_j\chi^{u_j})$, so $\wt{U}_1\cap\wt{U}_2\in\wt{\calf}$.

(4) This is immediate from the definition of $\wt{\calf}$. 

\GeorgeStory{Where George meets Paul.}

(5) Suppose $\wt{U} = \cup_{i=1}^m \wt{U}_i$ is in $\wt{\calf}$ with each $\wt{U}_i$ a $\Gamma$-admissible \conoidSet; we want to show that some $\wt{U}_i$ is in $\wt{\calf}$. By writing each $\wt{U}_i$ as a finite union of $\Gamma$-admissible cones, we may assume without loss of generality that each $\wt{U}_i$ is a $\Gamma$-admissible cone. Let $n_0 = 0$. Note that we can write
$$\wt{U_i} = \bigcap_{l=n_{i-1}+1}^{n_i} \wt{R}(1_\base, a_l \chi^{u_l}),$$
for some increasing sequence of integers $n_0<n_1<\cdots<n_m$ and terms $a_l\chi^{u_l}\in\base[\Mon]$. 
Since $\wt{U} \in\wtcalf$, there are terms $c_1\chi^{\mu_1},\ldots,c_q\chi^{\mu_q}\in\base[\Mon]$ such that  $\wt{U} \supseteq \cap_{j=1}^q \wt{R}(1_\base, c_j \chi^{\mu_j}),$ and $1_\base \geq |c_j \chi^{\mu_j}|_P$. For $1 \leq l \leq n_m$, set
$$b_l\chi^{\theta_l} = \begin{cases} a_l \chi^{u_l}& \text{ if } 1_\base \geq |a_l \chi^{u_l}|_P \\ a_l^{-1} \chi^{-u_l}& \text{ otherwise}   \end{cases}.$$
We have $1_\base \geq |b_l\chi^{\theta_l}|_P$ for $1 \leq l \leq n_m$, so the cone
$$\Paul = \bigcap_{l=1}^{n_m} \wt{R}(1_\base, b_l\chi^{\theta_l}) \cap \bigcap_{j=1}^q \wt{R}(1_\base, c_j \chi^{\mu_j})$$
is in $\wtcalf$. Also, $\Paul \subseteq \bigcap_{j=1}^q \wt{R}(1_\base, c_j \chi^{\mu_j}) \subseteq \wt{U}.$
To finish the proof it is enough to see the following claim: 
for any $i$, either $\Paul \subseteq \wt{U}_i$ or $\wt{U}_i \cap \relint(\Paul) = \emptyset$. For, if this claim holds, then the fact that $\Paul\subseteq \wt{U}=\dcup_{i=1}^m\wt{U}_i$ implies that $\Paul$ is contained in some $\wt{U}_{i_0}$.

To prove the claim, consider the linear
hyperplane arrangement in $\R_{\geq0}\times N_\mathbb{R}$ with hyperplanes
\begin{align*}
    &\{(r,\xi)\in \R_{\geq0}\times N_\mathbb{R} : r\log(a_l) + \left< \xi, u_l\right> = 0\} \text{ for all } l \text{ and } \\
    &\{(r,\xi)\in \R_{\geq0}\times N_\mathbb{R} : r\log(c_j) + \left< \xi, \mu_j\right> = 0\} \text{ for all } j.
\end{align*}
This hyperplane arrangement defines a fan $\Sigma$ such that $\Paul \in \Sigma$ and the support of $\Sigma$ is $\R_{\geq0}\times N_\mathbb{R}$. Moreover, for every $i$, $\wt{U}_i$ is the support of a subfan $\Sigma_i$ of $\Sigma$. Since $\wt{U}_i$ is the disjoint union of $\relint (\mathfrak{D})$ for $\mathfrak{D} \in \Sigma_i$ and $\Paul\in\Sigma$, if $\relint(\Paul)$ meets $\wt{U}_i$, then $\Paul \in \Sigma_i$, implying that $\Paul \subseteq \wt{U}_i$.

The proof for $\calf=\calf_P$ is analogous. 
\end{proof}


Theorem~\ref{thm:TheFilterIsAPrimeFilter} immediately gives us the following corollary.

\begin{coro}\label{coro:FilterDeterminedByCones/Polyhedra}
For prime congruences $P$ on $\base[\Mon]$, $\wtcalf_P$ is determined by which $\Gamma$-admissible cones it contains. For $P\in\ContBase{\base}\base[\Mon]$, $\calf_P$ is determined by which $\Gamma$-rational polyhedra it contains.
\end{coro}

\subsection{Flags and filters.}

We now directly relate the filters $\calf$ and $\wtcalf$ back to flags of polyhedra and cones. We start with a lemma that will facilitate our inductive argument.

\begin{lemma}\label{lemma:TruncationAndFilter}
Let $\calc_\bullet=(\calc_0\leq\cdots\leq\calc_k)$ be a simplicial flag of cones, let $P=P_{\calc_\bullet}$, and let $\wtcalf=\wtcalf_P$. For any $0\leq j<k$, consider the truncation $\calc_\bullet^{(j)}$, let $P'=P_{\calc_\bullet^{(j)}}$, and let $\wtcalf'=\wtcalf_{P'}$. Then $\wtcalf\subseteq\wtcalf'$.
\end{lemma}\begin{proof}
It suffices to show the result for $j=k-1$; the general result follows by induction. 

Recall that Proposition~\ref{prop:GeomEvalOfPNorm} gives explicit geometric conditions for when $1_{\kappa(P)}\geq|a\chi^u|_P$. 
For convenience, we now state the corresponding conditions in the case where $P$ is replaced by $P'$.
That is, $1_{\kappa(P')}\geq|a\chi^u|_{P'}$ if and only if 
\begin{align}\label{eq: conditions-rational-sets-geom-truncated}
    &\relint\calc_0\subseteq\wt{R}^\circ(1_\base, a \chi^{u}) \text{, or } \nonumber\\
    &\relint\calc_1\subseteq\wt{R}^\circ(1_\base, a \chi^{u})\text{, or}\nonumber\\
    &\qquad\qquad\quad \vdots\\
    &\relint\calc_{k-2}\subseteq\wt{R}^\circ(1_\base, a \chi^{u})\text{, or }\nonumber\\
    &\calc_{k-1}\subseteq\wt{R}(1_\base, a \chi^{u}).\nonumber
\end{align}
As before, we label the conditions in (\ref{eq: conditions-rational-sets-geom-truncated}) as (\ref{eq: conditions-rational-sets-geom-truncated}.0), (\ref{eq: conditions-rational-sets-geom-truncated}.1), $\ldots$, (\ref{eq: conditions-rational-sets-geom-truncated}.$k-2$), and (\ref{eq: conditions-rational-sets-geom-truncated}.$k-1$).

Suppose $\wt{U}\in\wtcalf$, so there are terms $a_1\chi^{u_1},\ldots,a_n\chi^{u_n}\in\base[\Mon]$ such that $\wt{U}\supseteq\dcap_{l=1}^n\wt{R}(1_\base,a_l\chi^{u_l})$ and $1_{\kappa(P)}\geq|a_l\chi^{u_l}|_P$ for $1\leq l\leq n$.
Note that, to get that $\wt{U}\in\wtcalf'$, it suffices to show that $1_{\kappa(P')}\geq|a_l\chi^{u_l}|_{P'}$ for each $1\leq l\leq n$. Fix such an $l$. 
Since $1_{\kappa(P)}\geq|a_l\chi^{u_l}|_P$, there is some $0\leq i_l\leq k$ such that $a_l\chi^{u_l}$ satisfies (\ref{eq: conditions-rational-sets-geom}.$i_l$).

If $i_l=k$ then $\calc_{k-1}\subseteq\calc_k\subseteq\wt{R}(1_\base, a_l \chi^{u_l})$, so $a_l\chi^{u_l}$ satisfies (\ref{eq: conditions-rational-sets-geom-truncated}.$k-1$), giving us $1_{\kappa(P')}\geq|a_l\chi^{u_l}|_{P'}$. If $i_l=k-1$ then, by Remark~\ref{rem:ResultsOfConditionsInProp:GeomEvalOfPNorm}, $\calc_{k-1}\subseteq\wt{R}(1_\base, a_l \chi^{u_l})$, i.e, $a_l\chi^{u_l}$ satisfies (\ref{eq: conditions-rational-sets-geom-truncated}.$k-1$), and so $1_{\kappa(P')}\geq|a_l\chi^{u_l}|_{P'}$ again.

So now suppose $i_l<k-1$. Then (\ref{eq: conditions-rational-sets-geom-truncated}.$i_l$) and (\ref{eq: conditions-rational-sets-geom}.$i_l$) are the same condition, so we get that $1_{\kappa(P')}\geq|a_l\chi^{u_l}|_{P'}$.
\end{proof}

\begin{theorem}\label{thm:TheFilterHasFlagMeaning}
For any prime $P$ on $\base[\Mon]$, let $\wtcalf=\wtcalf_P$ and pick a simplicial flag $\calc_\bullet$ of cones such that $P=P_{\calc_\bullet}$. Then, for any $\Gamma$-admissible cone $\wt{U}$ in $\R_{\geq0}\times N_{\R}$, $\wt{U}\in\wtcalf$ if and only if $\wt{U}$ is a neighborhood of $\calc_\bullet$.

If $P \in \ContBase{\base}{\base[\Mon]}$ then we can write $P=P_{\calp_\bullet}$ for some flag $\calp_\bullet$ of polyhedra and consider $\calf=\calf_P$. For any $\Gamma$-rational polyhedron $U$ in $N_\R$, $U\in \mathcal{F}$ if and only if $U$ is a neighborhood of the flag $\mathcal{P}_\bullet$.
\end{theorem}

\begin{proof}
We first consider $\wtcalf$ and $\calc_\bullet$. 
Suppose that $\wt{U}$ is a neighborhood of $\calc_\bullet$, so for $0\leq i\leq k$ we can pick a point $ w_i\in\wt{U}\cap\relint(\calc_i)$. The matrix $\matt=\begin{pmatrix} w_0\\\vdots\\ w_k\end{pmatrix}$ defines a homomorphism $\ph:\base[\Mon]\to\RnLexSemifield{k+1}$ and, by Definition~\ref{def:PrimeOfAFlag}, $P=\ker\ph$. Thus, for any $f,g\in\base[\Mon]$, $|f|_P\leq|g|_P$ if and only if $\ph(f)\leq \ph(g)$. To get that $\wt{U}\in\wtcalf$, we need to show that there are terms $a_1\chi^{u_1},\ldots,a_n\chi^{u_n}\in\base[\Mon]$ such that $\wt{U}\supseteq\wt{R}(1_\base,  a_1\chi^{u_1}, \ldots, a_n\chi^{u_n}) = \bigcap_{l=1}^n \wt{R}(1_\base, a_l\chi^{u_l})$ and $1_{\kappa(P)}\geq|a_l\chi^{u_l}|_P$.

Since $\wt{U}$ is $\Gamma$-admissible, there are $\fraku_1=(\alpha_1,u_1),\ldots,\fraku_n=(\alpha_n,u_n)\in\Gamma\times\Mon$ such that $\wt{U}=\{\fraku_1,\ldots,\fraku_n\}^\vee=\dcap_{l=1}^n\wt{R}(1_\base,a_l\chi^{u_l})$, where $a_l=t^{\alpha_l}\in\base^\times$. Since $\fraku_l$ is the exponent vector of $a_l\chi^{u_l}$, we have that $1_{\kappa(P)}\geq|a_l\chi^{u_l}|_P$ if and only if 
$$\begin{pmatrix} 0 \\ \vdots \\ 0\end{pmatrix}\geq_{lex}\matt\fraku_l=\begin{pmatrix} \angbra{ w_0,\fraku_l} \\ \vdots \\ \angbra{ w_k,\fraku_l}\end{pmatrix}.$$
Since each $ w_i\in\wt{U}=\{\fraku_1,\ldots,\fraku_n\}^\vee$, we have that each $\angbra{ w_i,\fraku_l}$ is non-positive, so we get that $1_{\kappa(P)}\geq|a_l\chi^{u_l}|_P$. Thus, $\wt{U}\in\wtcalf$.

\newcommand{\PaulBro}{\mathfrak{P}}
\newcommand{\paulbro}{\PaulBro}
For the other direction, suppose $\wt{U}\in\wtcalf$.  We proceed to prove that $\wt{U}$ is a neighborhood of $\calc_\bullet=(\calc_0\leq\cdots\leq\calc_k)$ by induction on $k$. For the base case of $k=0$, applying Proposition~\ref{prop:GeomEvalOfPNorm} shows that $\wt{U}\in\wtcalf$ gives us that there are  $a_1\chi^{u_1},\ldots,a_n\chi^{u_n}\in\base[\Mon]$ such that $\wt{U}\supseteq\dcap_{l=1}^n\wt{R}(1_\base,a_l\chi^{u_l})\supseteq\calc_0$, so $\wt{U}$ is a neighborhood of $\calc_\bullet=(\calc_0)$. For the inductive step, suppose that $k\geq1$ and the statement is true for $k-1$. 
Fix $a_1\chi^{u_1},\ldots,a_n\chi^{u_n}\in\base[\Mon]$ such that $\wt{U}\supseteq\dcap_{l=1}^n\wt{R}(1_\base,a_l\chi^{u_l})$ and $1_{\kappa(P)}\geq|a_l\chi^{u_l}|_P$. In particular, $\PaulBro:=\dcap_{l=1}^n\wt{R}(1_\base,a_l\chi^{u_l})\in\wtcalf$ and it suffices to show that $\PaulBro$ is a neighborhood of $\calc_\bullet$.
Consider the truncation $\calc_\bullet^{(k-1)}$ of $\calc_\bullet$, let $P'=P_{\calc_\bullet^{(k-1)}}$, and let $\wtcalf'=\wtcalf_{P'}$. Lemma~\ref{lemma:TruncationAndFilter} tells us that 
$\PaulBro\in\wtcalf'$
and so, by the inductive hypothesis, 
$\PaulBro$
is a neighborhood of $\calc_\bullet^{(k-1)}$. Since 
$\PaulBro$
is a neighborhood of $\calc_\bullet$ if and only if it is a neighborhood of $\calc_\bullet^{k-1}$ and meets $\calc_k\sdrop\calc_{k-1}$, it now suffices to show that 
$\PaulBro$
meets $\calc_k\sdrop\calc_{k-1}$.

\newcommand{\aletter}{j}
For each $1\leq l\leq n$, $1_{\kappa(P)}\geq|a_l\chi^{u_l}|_P$, so $a_l\chi^{u_l}$ satisfies (\ref{eq: conditions-rational-sets-geom}.$i_l$) for some $0\leq i_l\leq k$. If $i_l=k$ then $\calc_k\subseteq\wt{R}(1_\base,a_l\chi^{u_l})$, so 
$$
\PaulBro\cap(\calc_k\sdrop\calc_{k-1})
=\left(
\dcap_{\substack{
1\leq\aletter\leq n\\
\aletter\neq l}}
\wt{R}(1_\base,a_{\aletter}\chi^{u_{\aletter}})\right)\cap(\calc_k\sdrop\calc_{k-1}).
$$

Hence, 
it suffices to show that $\dcap_{\substack{1\leq\aletter\leq n\\\aletter\neq l}}\wt{R}(1_\base,a_{\aletter}\chi^{u_{\aletter}})$ meets $\calc_k\sdrop\calc_{k-1}$. Thus, we may assume without loss of generality that $0\leq i_l\leq k-1$ for all $l$, so $\relint\calc_{i_l}\subseteq\wt{R}^{\circ}(1_\base,a_l\chi^{u_l})$.

Since $\PaulBro$ is a neighborhood of $\calc_\bullet^{(k-1)}$ and $0\leq i_l\leq k-1$, we can pick a point $ w_l\in\relint(\calc_{i_l})\cap\PaulBro$, so, in particular, $ w_l\in\relint(\calc_{i_l})\subseteq\wt{R}^\circ(1_\base,a_l\chi^{u_l})$. 
Also, $ w_l\in\calc_k\cap\paulbro$. 
Fix $ v\in\calc_k\sdrop\calc_{k-1}$. Since $ w_l$ is in the interior of $\wt{R}(1_\base,a_l\chi^{u_l})$, there is some $\epsilon_l>0$ such that $ w_l+[0,\eps_l] v\subseteq\wt{R}(1_\base,a_l\chi^{u_l})$. Letting $\epsilon=\min_l \epsilon_l$, we get that $ w_l+\epsilon v\in\wt{R}(1_\base,a_l\chi^{u_l})$ for all $l$. Also, recall that each $ w_l\in\paulbro=\dcap_{\aletter=1}^n\wt{R}(1_\base,a_{\aletter}\chi^{u_{\aletter}})$.

Consider the point $\omega=\epsilon v+\dsum_{l=1}^n w_l$. If we fix $l$ then we can write $\omega=( w_l+\epsilon v)+\dsum_{\aletter\neq l} w_{\aletter}$ with each of $ w_l+\eps v$ and $ w_{\aletter}$ in the linear half-space $\wt{R}(1_\base,a_l\chi^{u_l})$, so $\omega\in\wt{R}(1_\base,a_l\chi^{u_l})$. Since this is true for all $l$, we conclude that $\omega\in\dcap_{l=1}^n\wt{R}(1_\base,a_l\chi^{u_l})=\paulbro$. Since $\dsum_{l=1}^n w_l\in\calc_k$ and $\eps v\in\calc_k\sdrop\calc_{k-1}$, we also have $\omega\in\calc_k\sdrop\calc_{k-1}$. So $\paulbro$ meets $\calc_k\sdrop\calc_{k-1}$, concluding the proof that $\wt{U}\in\wtcalf$ if and only if $\wt{U}$ is a neighborhood of $\calc_\bullet$.

Now suppose $P\in\ContBase{\base}\base[\Mon]$ and consider $\calp_\bullet$ and $\calf$ as in the statement of the theorem. Letting $\calc_\bullet=c(\calp_\bullet)$, we have $P=P_{\calc_{\bullet}}$. So, using the earlier part of the theorem, we have
\begin{align*}
U\in\calf&\iff c(U)\in\wtcalf\\
&\iff c(U)\text{ is a neighborhood of }\calc_\bullet=c(\calp_\bullet)\\
&\iff U\text{ is a neighborhood of }\calp_\bullet.
\end{align*}
\end{proof}

\begin{coro}\label{coro:FilterEqualIsFlagEquiv}
Consider simplicial flags $\calc_\bullet$ and $\calc_\bullet'$ of cones and let $P=P_{\calc_\bullet}$ and $P'=P_{\calc_\bullet'}$. Then $\wtcalf_{P}=\wtcalf_{P'}$ if and only if $\calc_\bullet$ and $\calc_\bullet'$ are locally equivalent.

If $P=P_{\calp_\bullet}$ and $P'=P_{\calp_\bullet'}$ for some flags $\calp_\bullet$ and $\calp_\bullet'$ of polyhedra, then $\calf_P=\calf_{P'}$ if and only if $\calp_\bullet$ and $\calp_\bullet'$ are locally equivalent.
\end{coro}\begin{proof}
This is immediate from Corollary~\ref{coro:FilterDeterminedByCones/Polyhedra} and Theorem~\ref{thm:TheFilterHasFlagMeaning}.
\end{proof}

Theorem~\ref{thm:TheFilterHasFlagMeaning} allows us to consider examples of primes in $\ContBase{\base}\base[\Mon]$ and understand their filters via flags $\calp_\bullet$ of polyhedra.

\begin{example} Let $\base = \Qmax$, $\Gamma = \Q$, and $\Mon = \Z^2$.

Let $C=\begin{pmatrix}
1&0&0
\end{pmatrix}$,
$C'=\begin{pmatrix}
1&\sqrt{2}&0
\end{pmatrix}$, and 
$C''=\begin{pmatrix}
1&\sqrt{2}&\sqrt{3}
\end{pmatrix}$, and let 
$\calp_\bullet$,
$\calp_\bullet'$, and 
$\calp_\bullet''$ 
be the corresponding flags of polyhedra in $N_{\R}=\R^2$. Let $P,P',$ and $P''$ be the corresponding prime congruences in $\ContBase{\Qmax}\Qmax[x^{\pm1},y^{\pm1}]$, and let $\calf,\calf'$, and $\calf''$ be the corresponding prime filters on the \BigLarry. Each of 
$\calp_\bullet$,
$\calp_\bullet'$, and 
$\calp_\bullet''$ consist of a single point whose coordinates are given by the second and third entries of the matrix. However, the corresponding filters look different. 

The filter $\calf$ consists of the zero-dimensional polyhedron $\{(0,0)\}$ and every rational polyhedral set containing it. On the other hand, $\calf'$ contains no zero-dimensional polyhedron: any rational polyhedron in $\R^2$ containing $(\sqrt{2},0)$ must contain an open interval on the $x$-axis. Similarly, $\calf''$ consists only of two-dimensional polyhedral sets. \exEnd\end{example}

\begin{example}\label{ex:Whale/Dolphin}

Still in the case where $\base = \Qmax$ and $\Mon = \Z^2$, we
now consider the matrices $C= \begin{pmatrix} 1  &\sqrt{2} &0\\ 0 & 1 &\sqrt{3}\end{pmatrix}$ and $C' = \begin{pmatrix} 1  &\sqrt{2} &0\\ 0 & 0 & 1\end{pmatrix}$ and let $\calp_\bullet$ and $\calp_\bullet'$ be the corresponding flags of polyhedra, which are depicted in Figure \ref{fig: whale-dolphin-1}. Then the corresponding filters $\calf$ and $\calf'$ are the same: they both consist of those rational polyhedral sets which contain $(r_1,r_2)\times[0,r_3)$ for some $r_1,r_2,r_3\in\R$ with $r_1<\sqrt{2}<r_2$ and $r_3>0$. One can see this visually depicted in Figures \ref{fig: whale-dolphin-2} and \ref{fig: whale-dolphin-3}. We will show in Corollary~\ref{coro:FlagsGiveSamePrimeIffLocallyEquiv} that this implies that $C$ and $C'$ define the same point $P\in\ContBase{\Qmax}\Qmax[x^{\pm1},y^{\pm1}]$. One can also check this by hand, by verifying that they give the same preorder on $\Terms{\Qmax[x^{\pm1},y^{\pm1}]}$, even though neither of $C,C'$ can be obtained from the other by downward gaussian elimination as in Lemma~\ref{lemma:rref}. \exEnd\end{example}

\begin{figure}[ht]
    \includegraphics[scale=1.48]{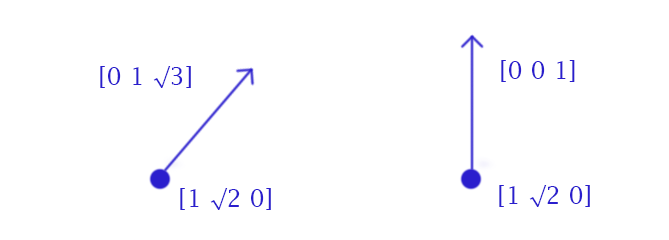}
    \caption{The flags $\calp_\bullet$ and $\calp_\bullet'$ of Example~\ref{ex:Whale/Dolphin}.}
    \label{fig: whale-dolphin-1}
\end{figure}
    
\begin{figure}[ht]
    \includegraphics[scale=1.48]{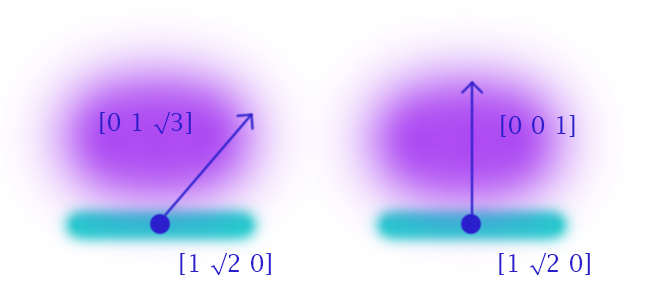} 
    \caption{This picture illustrates the filters corresponding to $\calp$ and $\calp'$ by drawing the ``infinitesimal neighborhoods'' that a rational polyhedron must contain to be a neighborhood of the flags. The turquoise represents the condition to be a neighborhood of the truncation $\calp_\bullet^{(0)}=(\calp_\bullet')^{(0)}$; the purple represents the added conditions to be neighborhoods of $\calp_\bullet$ and $\calp_\bullet'$.} 
    \label{fig: whale-dolphin-2}
\end{figure}

\begin{figure}[ht]
    \includegraphics[scale=1.48]{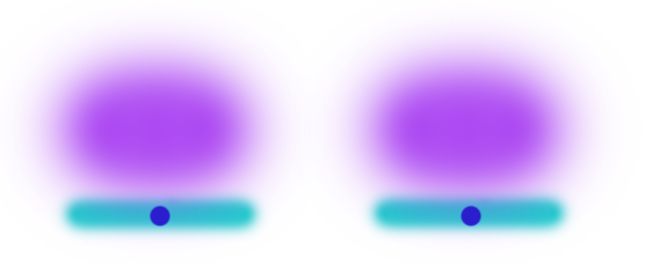}
    \caption{We see that the filters are equal, so the flags $\calp_\bullet$ and $\calp_\bullet'$ are locally equivalent.} 
    \label{fig: whale-dolphin-3}
\end{figure}

\begin{example}\label{ex:JellyfishAndFriends}
Consider the prime congruences on $\Qmax[x^{\pm1},y^{\pm1}]$ given by the matrices $C_1 = \begin{pmatrix} 1  &\sqrt{2} &\sqrt{3}\end{pmatrix}$,
$C_2 = \begin{pmatrix} 1  &0 &0\\ 0 & 1 &0\end{pmatrix}$, 
$C_3 = \begin{pmatrix} 1  &0 &0\\ 0 & 1 &\sqrt{2}\end{pmatrix}$, and 
$C_4 = \begin{pmatrix} 1  &0 &0\\ 0 & 1 &0\\ 0 & 0 &1\end{pmatrix}$. The corresponding flags of polyhedra and filters are pictured in Figure~\ref{fig: mix-george-points}.

\begin{figure}[ht]
    \includegraphics[scale=1.48]{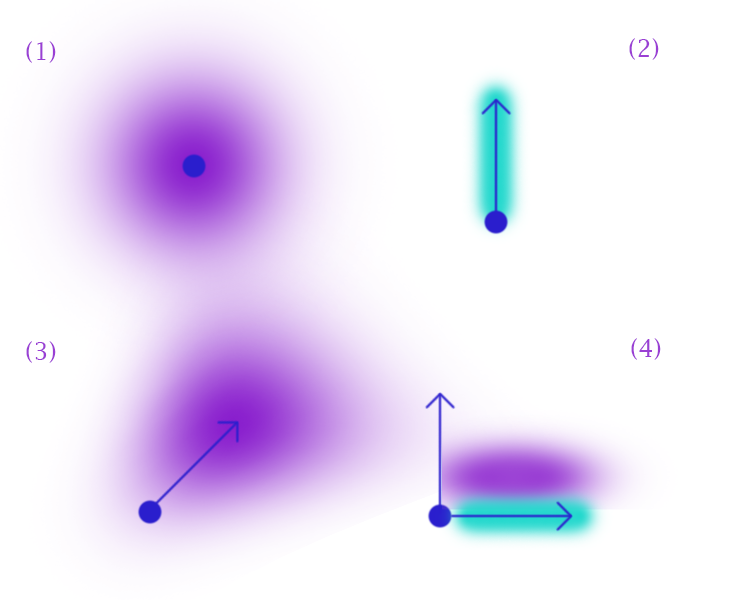}
    \caption{The flags of polyhedra and filters from Example~\ref{ex:JellyfishAndFriends}. In these pictures, turquoise represents ``one-dimensional infinitesimal neighborhoods'' and purple represents ``two-dimensional infinitesimal neighborhoods''. } 
    \label{fig: mix-george-points}
\end{figure} \exEnd\end{example}

\begin{remark}
We can formalize the notion of the dimension of the ``infinitesimal neighborhoods'' in the figures above by considering the smallest possible dimension of a polyhedron in $\calf_P$. See Proposition~\ref{prop: min-dim} for a way to compute this dimension as an algebraic invariant of $P$.
\end{remark}

\subsection{From filters back to congruences.}
We show that the points of $\ContBase{\base}{\base[\Mon]}$ are in bijection with prime filters on the lattice of $\Gamma$-rational polyhedral sets, where the filter contains some polytope. 

\begin{remark}\label{remark:FilterWithNoContPrime}
One may hope that there is a geometrically meaningful bijection between the points of $\ContBase{\base}{\base[\Mon]}$ and prime filters on the lattice of $\Gamma$-rational polyhedral sets. To see why this is not true, consider the following example: 
    \begin{align*}
        &N_\mathbb{R} = \R^2, \text{ and } \sigma = \{ (0, x)\in \R^2 : x\geq 0\} \\
        &\mathcal{F} = \{ U : \Gamma\text{-rational polyhedral sets, with } U \supseteq \sigma+(0,y) \text{ for some } y\in \R.\}
    \end{align*}
One can verify that $\mathcal{F}$ is a prime filter, however, it can only correspond to some point at infinity on the positive $y$-axis, and there is no such point in $\ContBase{\base}{\base[\Mon]}$.  
\end{remark}

The following lemma ensures that the prime filter $\calF$ of Remark~\ref{remark:FilterWithNoContPrime} does not occur as $\calf_P$ for any $P\in\ContBase{\base}\base[\Mon]$.

\begin{lemma}\label{lemma:FilterHasPolytope}
For any $P\in\ContBase{\base}\base[\Mon]$, $\calf_P$ contains a polytope.
\end{lemma}\begin{proof}
Write $P=P_{\calP_\bullet}$ for some flag $\calp_\bullet$ of polyhedra. Let $U$ be any full-dimensional $\Gamma$-rational polytope with the point $\calp_0$ in its interior. Then $U$ meets the relative interior of any polyhedron containing $\calp_0$,
so $U$ is a neighborhood of $\calp_\bullet$. By Theorem~\ref{thm:TheFilterHasFlagMeaning}, $U\in\calf_P$.
\end{proof}

\begin{defi}
A \textit{$\Gamma$-rational polytopal set} is a finite union of $\Gamma$-rational polytopes.
\end{defi}

The set of $\Gamma$-rational polytopal sets in $N_{\R}$ forms a lattice when ordered by inclusion. In this lattice, meet and join are given by intersection and union, respectively.

\begin{prop}\label{prop:BigAndLittleLarry}
There is a bijection between the set of prime filters on the lattice of $\Gamma$-rational polyhedral sets, where the filter contains some polytope and the set of prime filters on the lattice of $\Gamma$-rational polytopal sets.
\end{prop}

\begin{proof}
Given a prime filter $\F$ on the \LittleLarry, it generates a filter $\F'$ on the \BigLarry\ such that $\calf'$ contains a polytope, and a polytopal set is in $\calf$ if and only if it is in $\calf'$. Moreover, $\F'$ is also prime. Indeed, say $U = U_1 \cup U_2 \in \F'$ with $U_1$ and $U_2$ both $\Gamma$-rational. Letting $V\in\calf'$ be a polytope we have $U\cap V = (U_1 \cap V) \cup (U_2 \cap V)$ with $U\cap V$, $U_1 \cap V$, and $U_2 \cap V$ polytopal and $U\cap V\in\calf$. So either, $U_1 \cap V \in \F$ or $U_2 \cap V \in \F$, giving us that $U_1 \in \F'$ or $U_2 \in \F'$. 

\newcommand{\FilterRest}[1]{#1|_{\square}}
\newcommand{\calgrest}{\FilterRest{\calg}}
We will show that the above map $\calF\mapsto\calF'$ is a bijection by showing that its inverse is given by $\calg\mapsto \calgrest$ where, given a prime filter $\calg$ on the \BigLarry\ such that $\calg$ contains a polytope, $\calgrest$ denotes the restriction of $\calg$ to the \LittleLarry. We first note that, for any filter $\calf$ on the \LittleLarry, we have $\FilterRest{(\calf')}=\calf$ and, similarly, for any filter $\calg$ on the \BigLarry\  such that $\calg$ contains a polytope, we have $(\calgrest)'=\calg$.

It remains only to show that if $\calg$ is a prime filter on the \BigLarry\ such that $\calg$ contains a polytope, then the restriction of $\calg$ to the \LittleLarry\ is a prime filter $\F=\calgrest$. We check that each of the five conditions for a prime filter is satisfied.

(1) Since every $\Gamma$-rational polyhedron contains a $\Gamma$-rational polytope and $\calg$ is not all of the \BigLarry, $\calf$ is not all of the \LittleLarry.

(2) By hypothesis, $\calf$ is nonempty.

(3) If $U_1, U_2 \in \F$, then $U_1, U_2 \in \calg$ and $U_1, U_2$ are polytopal, so $U_1 \cap U_2 \in \calg$ and $U_1 \cap U_2$ is polytopal, implying that $U_1 \cap U_2 \in \F$. 

(4) If $U \in \F$ and $V$ is a $\Gamma$-rational polytopal set which contains $U$, then $V \supseteq U \in \calg$, implying that $V \in \calg$. Since $V$ is polytopal this gives $V \in \F$.

(5) If $U=\cup_{i=1}^m U_i\in\calf$ and each $U_i$ is a $\Gamma$-rational polytopal set, then $U \in \calg$ with each $U_i$ a $\Gamma$-rational polyhedral set, so some $U_{i_0}\in\calg$. Since $U_{i_0}$ is polytopal, we have $U_{i_0}\in\calf$.
\end{proof}

Given a prime filter $\wtcalf$ on the \ConoidLarry, we define a relation $\leq_{\wtcalf}$ on $\Terms{\base[\Mon]}$ by saying that $a\chi^u\leq_{\wtcalf}b\chi^\lambda$ if there is some $\wt{U}\in\wtcalf$ such that $\angbra{w,\begin{pmatrix}\log(a)\\u\end{pmatrix}}\leq\angbra{w,\begin{pmatrix}\log(b)\\\lambda\end{pmatrix}}$ for all $w\in\wt{U}$, i.e., $\wt{U}\subseteq\wt{R}(b\chi^\lambda,a\chi^u)$. Similarly, if $\calf$ is a filter on the \BigLarry, we have the relation $\leq_\calf$ on $\Terms{\base[\Mon]}$ defined by $a\chi^u\leq_\calf b\chi^\lambda$ if there is a $U \in \F$ such that $\log(a) + \left< \xi, u\right> \leq \log(b) + \left< \xi, \lambda\right>$ for all $\xi\in U$, that is, $U \subseteq R(b\chi^\lambda, a\chi^u)$. Note that, if $\wtcalf=\{c(U)\;:\;U\in\calf\}$, then $\leq_{\wtcalf}$ and $\leq_\calf$ are the same.

Since $\wtcalf$ is a filter, we have that $a\chi^u\leq_{\wtcalf}b\chi^\lambda$ if and only $\wt{R}(b\chi^\lambda,a\chi^u)\in\wtcalf$. Similarly, $a\chi^u\leq_\calf b\chi^\lambda$ if and only if $R(b\chi^\lambda,a\chi^u)\in\calf$.

\begin{prop}\label{prop: relationsF-P}
For any prime filter $\wtcalf$ on the \ConoidLarry, $\leq_{\wtcalf}$ is $\leq_P$ for some prime congruence $P$ on $\base[\Mon]$. 
For any prime filter $\F$ on the \BigLarry\ such that $\F$ contains a polytope, the relation $\leq_\F$ is $\leq_P$ for some $P \in \ContBase{\base}{\base[\Mon]}$.
\end{prop}

\begin{proof}
We begin by verifying that $\leq_{\wtcalf}$ satisfies conditions (1), (2), and (3) in Proposition~\ref{prop: side_claim}.

(1) ($\leq_{\wtcalf}$ is a total preorder) The relation is clearly reflexive. We need to show transitivity. Suppose $a\chi^u \leq_{\wtcalf} b\chi^\lambda$ and $b\chi^\lambda \leq_{\wtcalf} c\chi^\mu$, so there are $\wt{U}_1, \wt{U}_2 \in \wt{\F}$ such that 
\begin{align*}
    &\angbra{w,\begin{pmatrix}\log(a)\\u\end{pmatrix}}\leq\angbra{w,\begin{pmatrix}\log(b)\\\lambda\end{pmatrix}} \text{ for all } w\in \wt{U}_1 \text{ and } \\
    &\angbra{w,\begin{pmatrix}\log(b)\\\lambda\end{pmatrix}}\leq \angbra{w,\begin{pmatrix}\log(c)\\\mu\end{pmatrix}} \text{ for all } w \in \wt{U}_2.
\intertext{Then $\wt{U}_1 \cap \wt{U}_2 \in \wt{\F}$ and }
    &\angbra{w,\begin{pmatrix}\log(a)\\u\end{pmatrix}}\leq\angbra{w,\begin{pmatrix}\log(c)\\\mu\end{pmatrix}} \text{ for all } w\in \wt{U}_1\cap\wt{U}_2\\
\end{align*}
so $a\chi^u \leq_{\wtcalf} c\chi^\mu$. Thus, $\leq_{\wtcalf}$ is a preorder. To see that it is a total preorder, suppose that $a\chi^u \not\leq_{\wtcalf} b\chi^\lambda$. Then for every $\wt{U}\in \wtcalf$, $\wt{U}\not\subseteq\wt{R}(b\chi^\lambda,a\chi^u)$. We now consider two cases: either there is some $\wt{U}\in\wtcalf$ such that $\wt{U} \subseteq  \wt{R}( a\chi^u, b\chi^\lambda)$ or, for all $\wt{U}\in \wtcalf$ both $\wt{U} \cap \wt{R}( a\chi^u, b\chi^\lambda)$ and $\wt{U} \cap \wt{R}(b\chi^\lambda,  a\chi^u)$ are non-empty.
In the first case  $a\chi^u \geq_{\wtcalf} b\chi^\lambda$ and we are done. In the second case, fix any $\wt{U}\in \wtcalf$. Since $$\wt{U} = (\wt{U} \cap \wt{R}( a\chi^u, b\chi^\lambda)) \cup (\wt{U} \cap \wt{R}(b\chi^\lambda,  a\chi^u)),$$
we get that either $\wt{U} \cap \wt{R}( a\chi^u, b\chi^\lambda) \in \wtcalf$ or $\wt{U} \cap \wt{R}(b\chi^\lambda,  a\chi^u) \in \wtcalf$, because $\wtcalf$ is prime. However, since all $\wt{V} \in \wtcalf$ are not contained in $\wt{R}(b\chi^\lambda,a\chi^u)$, we must have  $\wt{U} \cap \wt{R}(a\chi^u, b\chi^\lambda) \in \wtcalf$. Thus, we have $a\chi^u \geq_{\wtcalf} b\chi^\lambda$.

(2) ($\leq_{\wtcalf}$ is multiplicative) Let  $a\chi^u, b\chi^\lambda$ and $ c\chi^\mu$ be terms in $\base[\Mon]$, with $a\chi^u \leq_{\wtcalf} b\chi^\lambda$. Then there exists $\wt{U} \in \wtcalf$ such that $\wt{U}\subseteq\wt{R}(b\chi^\lambda,a\chi^u)=\wt{R}(b\chi^\lambda c\chi^\mu,a\chi^u c\chi^\mu)$, so $a\chi^u c\chi^\mu\leq_{\wtcalf}b\chi^\lambda c\chi^\mu$.

(3) ($\leq_{\wtcalf}$ respects the order on $\base$) Let $a, b \in \base^\times$ such that $a \leq b$. Then every $\wt{U}\in\wtcalf$ is in $\wt{R}(b,a)=\R_{\geq0}\times N_{\R}$.

Thus, $\leq_{\wtcalf}$ is $\leq_P$ for some prime congruence $P$ on $\base[\Mon]$.

Now consider $\leq_{\calf}$. Applying the above to $\wtcalf=\{c(U)\;:\;U\in\calf\}$, we get that $\leq_{\calf}$, which is the same as $\leq_{\wtcalf}$, satisfies conditions (1), (2), and (3) in Proposition~\ref{prop: side_claim}. So it suffices to show that $\leq_{\calf}$ satisfies condition (4) from that proposition.

Consider a term $a\chi^u \in \base[\Mon]$. Fixing a polytope $U \in \F$, the set $\{ \log(a) + \left< \xi, u\right>\;:\; \xi\in U \}$ has a minimum value $\gamma \in \R$. Pick $\beta < \gamma$ in $\Gamma$ and set $b = t^\beta\in\base^\times$. Then $b <_\F a\chi^u$.
\end{proof}

In light of Proposition~\ref{prop: relationsF-P}, and because a prime congruence $P$ on $\base[\Mon]$ is determined by the relation $\leq_P$ on $\Terms{\base[\Mon]}$, we can make the following definitions.

\begin{defi}\label{def:PrimeGivenByFilter}
Given a prime filter $\wtcalf$ on the \ConoidLarry, the \emph{prime congruence $P_{\wtcalf}$ defined by $\wtcalf$} is the unique prime congruence $P$ on $\base[\Mon]$ such that $\leq_{\wtcalf}$ equals $\leq_P$. For any prime filter $\F$ on the \BigLarry\ such that $\F$ contains a polytope, the \emph{prime congruence $P_{\calf}$ defined by $\calf$} is the unique $P\in\ContBase{\base}\base[\Mon]$ such that $\leq_\F$ equals $\leq_P$.
\end{defi}

\begin{lemma}\label{lem: G-rational-Farkas}
Consider elements $a, a_1,a_2,\ldots,a_n\in\base^{\times}$ and $u,u_1, u_2,\ldots,u_n\in\Mon$, and suppose that $\displaystyle\bigcap_{l=1}^{n}R(1_{\base},a_l\chi^{u_l})$ is nonempty. Then $\displaystyle\bigcap_{l=1}^{n} R(1_{\base},a_l\chi^{u_l})\subseteq
R(1_{\base},a\chi^u)$ if and only if there are 
$m_1,m_2,\ldots,m_n\in\Z_{\geq0}$, $m\in\Z_{>0}$, and $b\in\base$
with $b\geq 1_{\base}$ such that $b(a\chi^u)^m=\displaystyle\prod_{l=1}^{n}(a_l\chi^{u_l})^{m_l}$.
\end{lemma}\begin{proof}
Note that $b(a\chi^u)^m=\displaystyle\prod_{l=1}^{n}(a_l\chi^{u_l})^{m_l}$ is equivalent to having 
$$\log b + m\log a  = \sum_{l=1}^n m_l \log a_l \text{ and } mu = \sum_{l=1}^n m_l u_l,$$
which is in turn equivalent to the existence of nonnegative rational numbers $r_1, \ldots, r_n$, such that 
$$\log a  \leq \sum_{l=1}^n r_l\log a_l \text{ and } u = \sum_{l=1}^n r_l u_l.$$
By \cite[Proposition 2.3]{Fri19} applied with the extension of ordered fields $\R/\Q$, this happens exactly if 
$\langle x,-u\rangle\geq\log a$ is valid for all $x$ in 
$$\{x\in N_{\R}\ :\ \langle x,-u_l \rangle\geq\log a_l\text{, for }1\leq l\leq n\}=\displaystyle\bigcap_{l=1}^n R(1_{\base}, a_l\chi^{u_l}),$$
i.e., if 
$\displaystyle\bigcap_{l=1}^{n}R(1_{\base}, a_l\chi^{u_l})\subseteq R(1_{\base},a\chi^u)$.
\end{proof}

\begin{remark}\label{remark:ThisFailsForCones}
The version of Lemma~\ref{lem: G-rational-Farkas} with $\wt{R}$ in place of $R$ is false. For example, let $\base=\T$ and $\Mon=\Z$. Then $\wt{R}(1_\T,t^1x^{-1})\cap\wt{R}(1_\T,t^1)\subseteq\wt{R}(1_\T,t^1x^{-2})$, but there are no $m_1,m_2\in\Z_{\geq0}$, $m\in\Z_{>0}$, and $b\in\T$ with $b\geq 1_{\T}$ such that $b(t^1x^{-2})^m=(t^1x^{-1})^{m_1}(t^1)^{m_2}$. 
\end{remark}

In light of Remark~\ref{remark:ThisFailsForCones}, the proof of the following proposition requires case-work.

\begin{prop}\label{prop:IneqInPrimeIffHalfspaceInItsFilter}
For any prime congruence $P$ on $\base[\Mon]$, $\wt{R}(1_\base,a\chi^u)\in\wtcalf_P$ if and only if $1_{\kappa(P)}\geq|a\chi^u|_P$. If $P\in\ContBase{\base}\base[\Mon]$ then $R(1_\base,a\chi^u)\in\calf_P$ if and only if $1_{\kappa(P)}\geq|a\chi^u|_P$.
\end{prop}\begin{proof}
The ``if'' directions in the above statements follow immediately from the definitions of $\wtcalf_p$ and $\calf_P$. For the other directions, note that the statements are trivially true if $a=0_\base$, so assume $a\neq0_\base$.

Suppose $P\in\ContBase{\base}\base[\Mon]$ and $R(1_\base,a\chi^u)\in\calf_P$. So there are terms $a_1\chi^{u_1},\ldots,a_n\chi^{u_n}$ such that $R(1_\base, a\chi^u) \supseteq \cap_{l=1}^n R(1_\base, a_l\chi^{u_l})$ and $|a_l\chi^{u_l}|_{P} \leq 1_{\kappa(P)}$ for each $l$. Lemma~\ref{lemma:RatSetsAreGammaRat/Ad} tells us that $\cap_{l=1}^n R(1_\base, a_l\chi^{u_l})$ is nonempty. So, by Lemma~\ref{lem: G-rational-Farkas}, there are $m_1,m_2,\ldots,m_n\in\Z_{\geq0}$, $m\in\Z_{>0}$, and $b\in\base$ with $b\geq1_\base$ such that $b(a\chi^u)^m=\displaystyle\prod_{l=1}^{n}(a_l\chi^{u_l})^{m_l}$. Then $b^{-1}\leq 1_{\base}$ and 
$$|a\chi^u|_P^m=|b^{-1}|_P\dprod_{l=1}^n|a_l\chi^{u_l}|_P^{m_l}\leq 1_{\kappa(P)},$$
so $|a\chi^u|_P\leq 1_{\kappa(P)}$.

Now consider a general prime congruence $P$ on $\base[\Mon]$, and suppose $\wt{R}(1_\base,a\chi^u)\in\wtcalf_P$. There are terms $a_1\chi^{u_1},\ldots,a_n\chi^{u_n}$ such that $\wt{R}(1_\base, a\chi^u) \supseteq \cap_{l=1}^n \wt{R}(1_\base, a_l\chi^{u_l})$ and $|a_l\chi^{u_l}|_{P} \leq 1_{\kappa(P)}$ for each $l$. Let $\calc_\bullet=(\calc_0\leq\cdots\leq\calc_k)$ be a simplicial flag of cones such that $P=P_{\calc_\bullet}$. Theorem~\ref{thm:TheFilterHasFlagMeaning} tells us that there is some $\wt{U}\in\wtcalf_P$ with $\wt{U}\subseteq\{0\}\times N_{\R}$ if and only if $\calc_k\subseteq\{0\}\times N_{\R}$. We consider two cases, as to whether this happens or not.

Suppose that $\calc_k\not\subseteq\{0\}\times N_{\R}$. Since $\wt{R}(1_\base,a_1\chi^{u_1},\ldots,a_n\chi^{u_n})=\cap_{l=1}^n \wt{R}(1_\base, a_l\chi^{u_l})$  and $\wt{R}(1_\base,a\chi^u)$ are in $\wtcalf_P$, neither of them is contained in $\{0\}\times N_{\R}$. Therefore, $\wt{R}(1_\base,a\chi^u)=c(R(1_\base,a\chi^u))$ and $\wt{R}(1_\base,a_1\chi^{u_1},\ldots,a_n\chi^{u_n})=c(R(1_\base,a_1\chi^{u_1},\ldots,a_n\chi^{u_n}))$. So $$R(1_\base,a\chi^u)\supseteq R(1_\base,a_1\chi^{u_1},\ldots,a_n\chi^{u_n})=\cap_{l=1}^n R(1_\base, a_l\chi^{u_l})$$ with $\cap_{l=1}^n R(1_\base, a_l\chi^{u_l})$ nonempty. We can now apply Lemma~\ref{lem: G-rational-Farkas} and, as in the previous case, conclude that $|a\chi^u|_P\leq 1_{\kappa(P)}$.

Finally, suppose that $\calc_k\subseteq\{0\}\times N_{\R}$. Then, by the definition of $P_{\calc_\bullet}$, $|b|_P=1_{\kappa(P)}$ for any $b\in\base^\times$. Note that $\wt{R}(1_\base,a\chi^u)\cap(\{0\}\times N_{\R})=\{0\}\times u^\vee$, where $u^\vee=\{\zeta\in N_{\R}\;:\;\angbra{\zeta,u}\leq0_{\R}\}$. Similarly, $\cap_{l=1}^n\wt{R}(1_\base,a\chi^u)\cap(\{0\}\times N_{\R})=\{0\}\times \{u_1,\ldots,u_n\}^\vee$, so we have that $\{u_1,\ldots,u_n\}^\vee\subseteq u^\vee$. So, by the $\Q$-version of \cite[Proposition 1.9]{Zieg95},
there are $m_1,m_2,\ldots,m_n\in\Z_{\geq0}$ and $m\in\Z_{>0}$ such that $mu=\sum_{l=1}^n m_lu_l$. Thus,
$$|a\chi^u|_P^m=|\chi^u|_P^m=\dprod_{l=1}^n|\chi^{u_l}|_P^{m_l}=\dprod_{l=1}^n|a_l\chi^{u_l}|_P^{m_l}\leq 1_{\kappa(P)},$$
so $|a\chi^u|_P\leq 1_{\kappa(P)}$.
\end{proof}

\begin{theorem}\label{thm:BijectionFiltersAndCont}
The map $P\mapsto\wtcalf_P$ gives a bijection from the set of prime congruences on $\base[\Mon]$ to the set of prime filters on the \ConoidLarry. The inverse map is given by $\wtcalf\mapsto P_{\wtcalf}$.

The map $P\mapsto\calf_P$ gives a bijection from $\ContBase{\base}{\base[\Mon]}$ to the set of prime filters on the \BigLarry\ where the filter contains some polytope. The inverse map is given by $\calf\mapsto P_\calf$.
\end{theorem}

\begin{proof}
Note that, by Theorem~\ref{thm:TheFilterIsAPrimeFilter}, Lemma~\ref{lemma:FilterHasPolytope}, and Definition~\ref{def:PrimeGivenByFilter}, all of these maps are well-defined.

We show that $P\mapsto\calf_P$ and $\calf\mapsto P_{\calf}$ are inverses. The proof for $P\mapsto\wtcalf_P$ and $\wtcalf\mapsto P_{\wtcalf}$ is analogous.

To show that $P=P_{\calf_P}$ for any $P\in\ContBase{\base}\base[\Mon]$, it suffices to show that $a\chi^u\leq_P b\chi^\lambda$ if and only if $a\chi^u\leq_{P_{\calF_P}} b\chi^\lambda$ for any terms $a\chi^u, b\chi^\lambda\in\base[\Mon]$. Since $b\chi^\lambda$ is a unit in $\base[\Mon]$, we can divide both inequalities by $b\chi^\lambda$, and so we may assume without loss of generality that $b\chi^\lambda=1_\base$. By Proposition~\ref{prop:IneqInPrimeIffHalfspaceInItsFilter}, we know that $a\chi^u\leq_P 1_\base$ if and only if $R(1_\base,a\chi^u)\in\calf_P$. 
By the definitions of $P_{\calf_P}$ and $\leq_{\calf_P}$, 
$R(1_\base,a\chi^u)\in\calf_P$ if and only if $a\chi^u\leq_{\calf_P}1_\base$ which, in turn, is the same as having $a\chi^u\leq_{P_{\calf_P}}1_\base$.

Now fix a prime filter $\calf$ on the \BigLarry\ such that $\calf$ contains a polytope; we wish to show that $\calf=\calf_{P_\calf}$. Since $\calf$ and $\calf_{P_\calf}$ are prime filters on the \BigLarry, they are determined by which $\Gamma$-rational polyhedra they contain. Using the fact that $\calf$ and $\calf_{P_\calf}$ are both filters, we then see that they are determined by which $\Gamma$-rational half-spaces they contain. But every $\Gamma$-rational half-space can be written as $R(1_\base,a\chi^u)$ for some term $a\chi^u\in\base[\Mon]$, so it suffices to show that $R(1_\base,a\chi^u)\in\calf$ if and only if $R(1_\base,a\chi^u)\in\calf_{P_\calf}$. By Proposition~\ref{prop:IneqInPrimeIffHalfspaceInItsFilter}, $R(1_\base,a\chi^u)\in\calf_{P_\calf}$ if and only if $ a\chi^u \leq_{P_\calf} 1_\base$. By the definitions of $P_{\calf}$ and $\leq_{\calf}$, this happens if and only if $R(1_\base,a\chi^u)\in\calf$.
\end{proof}


\begin{coro}\label{coro:FlagsGiveSamePrimeIffLocallyEquiv}
Let $\calc_\bullet$ and $\calc_\bullet'$ be simplicial flags of cones. Then $P_{\calc_\bullet}=P_{\calc_\bullet'}$ if and only if $\calc_\bullet$ and $\calc_\bullet'$ are locally equivalent.

If $P=P_{\calp_\bullet}$ and $P'=P_{\calp_\bullet'}$ for some flags $\calp_\bullet$ and $\calp_\bullet'$ of polyhedra, then $P=P'$ if and only if $\calp_\bullet$ and $\calp_\bullet'$ are locally equivalent.
\end{coro}\begin{proof}
By Theorem~\ref{thm:BijectionFiltersAndCont}, $P=P'$ if and only if $\wtcalf_P=\wtcalf_{P'}$. Corollary~\ref{coro:FilterEqualIsFlagEquiv} tells us that this happens exactly if $\calc_\bullet$ and $\calc_\bullet'$ are locally equivalent.

The proof of the second statement is analogous.
\end{proof}

\begin{coro}\label{coro:explicitBijectionsContAndThreeOthers} Let $\base$ be a sub-semifield of $\T$ and let $\Mon\cong\Z^n$.
There are explicit bijections between
\begin{enumerate}
    \item $\ContBase{\base}\base[\Mon]$,
    \item the set of flags $\calp_\bullet$ of polyhedra in $N_{\R}$ modulo local equivalence,
    \item the set of prime filters on the \BigLarry\ such that the filter contains a polytope, and
    \item the set of prime filters on the \LittleLarry.
\end{enumerate}
These bijections are given as follows. The map from (2) to (1) sends a flag $\calp_\bullet$ to the prime congruence 
$P_{\calp_\bullet'}$ where $\calp_\bullet'$ is any flag of polyhedra which is locally equivalent to $\calp_\bullet$ and for which $c(\calp_\bullet')$ is simplicial.
The map from (1) to (3) is $P\mapsto\calf_P$. The map from (3) to (4) sends $\calf$ to 
the filter
$\{U\in\calf\;:\;U\text{ is polytopal}\}$.
\end{coro}\begin{proof}
This follows immediately from
Proposition~\ref{prop:EveryFlagLocEquivToSimplicial},
Proposition~\ref{prop:BigAndLittleLarry}, Theorem~\ref{thm:BijectionFiltersAndCont}, and Corollary~\ref{coro:FlagsGiveSamePrimeIffLocallyEquiv}.
\end{proof}

\begin{coro}\label{coro:explicit BijectionsAllPrimeCongsAndThreeOthers}
Let $\base$ be a sub-semifield of $\T$ and let $\Mon\cong\Z^n$.
There are explicit bijections between
\begin{enumerate}
    \item the set of all prime congruences on $\base[\Mon]$,
    \item the set of flags $\calc_\bullet$ of cones in $\R_{\geq0}\times N_{\R}$ modulo local equivalence,
    \item the set of prime filters on the \ConoidLarry.
\end{enumerate}
These bijections are given as follows. The map from (2) to (1) sends a flag $\calc_\bullet$ to the prime congruence 
$P_{\calc_\bullet'}$, where $\calc_\bullet'$ is any simplicial flag of cones which is locally equivalent to $\calc_\bullet$.
The map from (1) to (3) is $P\mapsto\wtcalf_P$.
\end{coro}\begin{proof}
This follows immediately from
Proposition~\ref{prop:EveryFlagLocEquivToSimplicial},
Theorem~\ref{thm:BijectionFiltersAndCont}, and  Corollary~\ref{coro:FlagsGiveSamePrimeIffLocallyEquiv}.
\end{proof}

\begin{coro}\label{coro:GeomMeaningOfOtherPolyhedralFilters}
There is a bijection from the set of prime congruences $P$ on $\base[\Mon]$ such that the map $\base\to\base[\Mon]\to\kappa(P)$ is injective and the set of prime filters on the \BigLarry. 
This bijection extends the bijection $P\mapsto\calf_P$ from $\ContBase{\base}\base[\Mon]$ to the set of prime filters on the \BigLarry\ such that the filter contains a polytope.
\end{coro}
\begin{proof}
Note that the map $\calf\mapsto\{c(U)\;:\;U\in\calf\}$ is a bijection from the set of prime filters on the \BigLarry\ to the set of prime filters $\wtcalf$ on the \ConoidLarry\ such that every $\wt{U}\in\wtcalf$ is not contained in $\{0\}\times N_{\R}$. So it suffices to show that, if $P$ is a prime congruence on $\base[\Mon]$, the map $\base\to\base[\Mon]\to\kappa(P)$ is injective if and only if every $\wt{U}\in\wtcalf_P$ is not contained in $\{0\}\times N_{\R}$. 

Pick a simplicial flag $\calc_\bullet=(\calc_0\leq\cdots\leq\calc_k)$ of cones such that $P=P_{\calc_\bullet}$ and fix any $w_i\in\relint\calc_i\subseteq\calc_i\sdrop\calc_{i-1}$ for $0\leq i\leq k$. Thus, $\begin{pmatrix}w_0\\\vdots\\w_k\end{pmatrix}$ is a defining matrix for $P$, and so the map $\base\to\base[\Mon]\to\kappa(P)$ is injective if and only if the first column of this matrix is not all zero, i.e., if $w_i\notin\{0\}\times N_{\R}$ for some $i$. Since $w_i\in\relint\calc_i$, $w_i\notin\{0\}\times N_{\R}$ if and only if $\calc_i\not\subseteq\{0\}\times N_{\R}$. So $\base\to\base[\Mon]\to\kappa(P)$ is injective if and only if $\calc_k\not\subseteq\{0\}\times N_{\R}$. 

But, by Theorem~\ref{thm:TheFilterHasFlagMeaning}, $\calc_k\not\subseteq\{0\}\times N_{\R}$ if and only if every $\wt{U}\in\wtcalf_P$ is not contained in $\{0\}\times N_{\R}$.
\end{proof}

\begin{remark}
Under the bijection of Corollary~\ref{coro:GeomMeaningOfOtherPolyhedralFilters}, the prime filter given in Remark~\ref{remark:FilterWithNoContPrime} corresponds to the prime congruence on $\T[x^{\pm1},y^{\pm1}]$ with defining matrix $\begin{pmatrix}
0&0&1\\
1&0&0
\end{pmatrix}$.
\end{remark}

\begin{coro}\label{cor: classify-mon-orders} Let $\Mon\cong\Z^n$.
There are explicit bijections between
\begin{enumerate}
    \item the set of all prime congruences on $\B[\Mon]$,
    \item the set of flags $\calc_\bullet$ of cones in $N_{\R}$ modulo rational local equivalence,
    \item the set of prime filters on the lattice of rational fan support sets in $N_{\R}$,
    \item the set of all monomial preorders on a Laurent polynomial ring in $n$ variables.
\end{enumerate}
The bijections between (1), (2), and (3) are given by applying the bijections of Corollary~\ref{coro:explicit BijectionsAllPrimeCongsAndThreeOthers} as follows. Given a prime congruence $P$ on $\B[\Mon]$, view it as a prime congruence on $\T[\Mon]$ by pulling back along the the map $\T[\Mon]\to\B[\Mon]$. Given a flag of cones in $N_{\R}$, consider it as a flag of cones in $\R_{\geq0}\times N_{\R}$ by adding a first coordinate $0$. Given a prime filter on the lattice of rational fan support sets in $N_{\R}$, push it forward to a filter base on the lattice of $\R$-admissible fan support sets in $\R_{\geq0}\times N_{\R}$ via the map $N_{\R}\to\R_{\geq0}\times N_{\R}$ adjoining a first coordinate $0$, and let this filter base generate a filter. The map from (1) to (4) is given by sending a prime congruence $P$ to the preorder $\leq_P$ on $\Terms{\B[\Mon]}=\Mon$.

\end{coro}\begin{proof}
The primes on $\T[\Mon]$ which are pullbacks of primes on $\B[\Mon]$ are those that are given by matrices with first column zero. This shows that the primes on $\B[\Mon]$ are in bijection with local equivalence classes of flags of cones contained in $\{0\} \times N_{\R}$ modulo local equivalence, i.e., flags of cones in $N_{\R}$ modulo rational local equivalence. By Corollaries~\ref{coro:explicit BijectionsAllPrimeCongsAndThreeOthers} and \ref{coro:GeomMeaningOfOtherPolyhedralFilters}, these are in bijection with the set of those prime filters on the lattice of $\R$-admissible fan support sets whose restriction to $\{1\} \times N_{\R}$ is not a prime filter because this restriction contains the empty set. The only way this can occur is if the filter contains a set that is contained in $\{0\} \times N_{\R}$. This, in turn, happens if and only if the filter is generated by the pushforward of a prime filter on the lattice of rational fan support sets in $N_{\R}$ along the inclusion map $N_{\R}\to\R_{\geq0}\times N_{\R}$ of the zero-slice. 
\end{proof}

The following corollary is an application of the above results. It provides an explicit criterion for when two matrices define the same prime.

\begin{coro}\label{coro:DoTheseMatricesGiveTheSamePrime}
Let $C$ and $C'$ be real matrices of sizes $k\times n$ and $k'\times n$, respectively. Suppose that $C$ and $C'$ define prime congruences on $\base[x_1^{\pm1},\ldots,x_n^{\pm1}]$, i.e., the first columns of $C$ and $C'$ are lexicographically at least the zero vector. Use downward gaussian elimination and removal of rows of zeros on $C$ and $C'$ to obtain matrices $\wt{C}$ and $\wt{C}'$, respectively, whose first columns have all entries non-negative and whose rows are linearly independent. Then $C$ and $C'$ define the same prime congruence on $\base[x_1^{\pm1},\ldots,x_n^{\pm1}]$ if and only if the flags $\calc_\bullet(\wt{C})$ and $\calc_\bullet(\wt{C}')$ of cones are locally equivalent.
\end{coro}\begin{proof}
By Lemma~\ref{lemma:rref}, the prime congruences that $C$ and $C'$ define are $P_{\calc_\bullet(\wt{C})}$ and $P_{\calc_\bullet(\wt{C}')}$, respectively. By Corollary~\ref{coro:FlagsGiveSamePrimeIffLocallyEquiv}, these are equal if and only if $\calc_\bullet(\wt{C})$ and $\calc_\bullet(\wt{C}')$ are locally equivalent.
\end{proof}

We now use Proposition~\ref{prop:IneqInPrimeIffHalfspaceInItsFilter} to show that the smallest dimension of a polyhedron in $\calf_P$ is an algebraic invariant of $P \in \ContBase{\base}{\base[\Mon]}$. 

\begin{proposition}\label{prop: min-dim}
For any $P \in \ContBase{\base}{\base[\Mon]}$,
$$\min \{ \dim U : U \in \calf_P\} 
= \rk (\kappa(P)^\times/\base^\times) 
= \rk(M)-\HT(P)=\trdeg(\kappa(P)/\base).$$
\end{proposition}
\begin{proof}

Lemma~\ref{lemma:HeightInToricCase} gives us the second equality in the proposition. By \cite[Proposition~B.11]{FM22}, we get the equality with $\trdeg(\kappa(P)/\base)$.

Note that $\min \{ \dim U : U \in \calf_P\}$ is attained for some $\Gamma$-rational polyhedron $U\in\calf_P$, for which $\dim U$ is the dimension of the smallest $\Gamma$-rational affine space that contains $U$. 
In particular, if we let $\calH$ be the subgroup of $\Gamma\times\Mon$ consisting of those $(\alpha,u)$ such that 
the hyperplane $\{\xi\in N_{\R}\;:\;\alpha+\angbra{\xi,u}=0\}$
is in $\calf_P$, then $\min \{ \dim U : U \in \calf_P\}=\dim N_{\R}-\rk\calH$.
Note that, for any term $a\chi^u\in\base[\Mon]$, 
we have that $(\log(a),u)$ is in $\calh$ 
exactly if $R(1_\base,a\chi^u)$ and $R(1_\base,a^{-1}\chi^{-u})$ are both in $\calf_P$. By  Proposition~\ref{prop:IneqInPrimeIffHalfspaceInItsFilter}, this happens exactly if $1_{\kappa(P)}=|a\chi^u|_P$.

Choose a defining matrix for $P \in \ContBase{\base}{\base[\Mon]}$ of the form
$$\begin{pmatrix}1 & \xi_0\\0 &\xi_1 \\ \vdots & \vdots  \\ 0 & \xi_k \end{pmatrix}$$
for some $\xi_0,\ldots,\xi_k\in N_{\R}$.
Then, looking at the conditions (\ref{eq: conditions-rational-sets}.0),\ldots,(\ref{eq: conditions-rational-sets}.$k$) from Section~\ref{subsec:FilterOfAPrimeCong}, we see that $1_{\kappa(P)}=|a\chi^u|_P$ if and only if 
$$log(a)+ \left< \xi_0, u\right> = \left< \xi_1, u\right>= \cdots =\left< \xi_{k-1}, u\right>=\left< \xi_{k}, u\right>=0.$$
We now see that the projection $\Gamma \times M \rightarrow M$ maps $\calH$ isomorphically onto its image. Indeed, the values of $\xi_0$ and $u$ uniquely determine $\log(a)$. So 
\begin{align*}
    \rk(\calH)&=\rk \left\{ u\in M \;:\; \left< \xi_0, u\right> \in \Gamma \text{ and } \left< \xi_i, u\right> = 0, \text{ for } 1 \leq i \leq k \right\} \\
    &= \rk (\ker (M \rightarrow \kappa(P)^\times/\base^\times)) = \rk M-\rk(\kappa(P)^\times/\base^\times),
\end{align*}
where the final equality holds because the map $M\to\kappa(P)^\times/\base^\times$ is surjective. Thus, 
\begin{align*}
    \min \{ \dim U \;:\; U \in \calf_P\} &= \dim N_\R - \rk\calH = \rk M - \rk (\ker (M \rightarrow \kappa(P)^\times/\base^\times)) \\  &= \rk \kappa(P)^\times/\base^\times.
\end{align*}

\end{proof}

The previous result finally allows us to prove the geometric criterion of when $\leq_P$ is a valuated monomial \emph{order}.
A flag $\calp_\bullet=(\calp_0\leq\calp_1\leq\cdots\leq\calp_k)$ of polyhedra in $N_{\R}$ is called \emph{complete} if $k=\dim N_{\R}$.

\begin{prop}\label{prop:whenDoesAPrimeTotallyOrder}
For any $P \in \ContBase{\base}{\base[\Mon]}$, $\leq_P$ is an \emph{order} on $\Terms{\base[\Mon]}$ if and only if we can write $P=P_{\calp_\bullet}$ with $\calp_\bullet$ a complete flag in $N_{\R}$.
\end{prop}
\begin{proof}
If $\calp_\bullet$ is complete, then any $\Gamma$-rational neighborhood of $\calp_\bullet$ must be full-dimensional because it contains $\dim N_{\R}+1$ affinely independent points. 
So $\dim_{N_{\R}}=\min \{ \dim U : U \in \calf_P\}=\rk(M)-\HT(P)$. Thus $\HT(P)=0$ and so, by \cite[Corollary~9.15]{FM22} $\leq_P$ is an order on $\Terms{\base[\Mon]}$.

For the other direction, suppose that $\leq_P$ is an order on $\Terms{\base[\Mon]}$. Let $\calp_\bullet=(\calp_0\leq\calp_1\leq\cdots\leq\calp_k)$ be a flag of polyhedra in $N_{\R}$ such that $P=P_{\calp_\bullet}$ and $k$ is as large as possible. Now assume, for contradiction, that $k<\dim N_{\R}$. 
For $0\leq i\leq k$ pick $\xi_i\in\calp_i\sdrop\calp_{i-1}$, so 
$C=\begin{pmatrix}1&\xi_0\\
1&\xi_1\\
\vdots&\vdots\\
1&\xi_k
\end{pmatrix}$
is a defining matrix for $P$. Since $k<\dim N_{\R}$, there is a point $\xi_{k+1}$ which is affinely independent of $\xi_0,\ldots,\xi_k$. So the matrix 
$C'=\left(\begin{array}{cc}
\multicolumn{2}{c}{\hspace{-2mm}C}\\\hline
1&\xi_{k+1}
\end{array}\right)$
defines a flag $\calp_\bullet '=(\calp_0'\leq\calp_1'\leq\cdots\leq\calp_{k+1}')$ of simplices in $N_{\R}$.
For any two exponent vectors $\fraku_1$ and $\fraku_2$ we have $C'\fraku_1\leq_{\lex}C'\fraku_2$ if and only if $C\fraku_1\leq_{\lex}C\fraku_2$ because $\leq_P$ is an order. Thus $P_{\calp_{\bullet}'}=P$, contradicting the maximality of $k$.
\end{proof}

\begin{remark}
The same proof 
shows that the statement of Proposition~\ref{prop:whenDoesAPrimeTotallyOrder} is also true when $\Mon$ is any toric monoid.
\end{remark}


\renewcommand{\bibliofont}{\small}
\bibliographystyle{alpha}

\end{document}